\date{\today}

\documentclass[10pt]{article}

\usepackage[utf8]{inputenc}
\usepackage[T1]{fontenc}

\usepackage{amsmath}
\usepackage{amsthm}
\usepackage{amssymb}
\usepackage{url}
\usepackage{latexsym}
\usepackage{enumitem}

\usepackage{musicography}
\newcommand{\gridded}{{\scalebox{0.7}{\musSharp}}}
\newcommand{\ggridded}{{\musDoubleSharp}}
\newcommand{\colorder}{\trianglelefteq^{\text{col}}}%
\newcommand{\roworder}{\trianglelefteq^{\text{row}}}%

\usepackage{mathtools} 

\usepackage{shuffle} 

\usepackage{tikz}
\usetikzlibrary{arrows, automata, decorations.pathreplacing, fit, matrix, patterns, positioning, calc}
\usepackage{tikz-qtree}
\usepackage{xifthen} 

\usepackage[small,it,singlelinecheck=false]{caption}

\usepackage[small, pagestyles]{titlesec}
\usepackage{units} 
\usepackage{verbatim}

\usepackage[square,comma,numbers,sort&compress]{natbib}

\makeatletter
\renewenvironment{proof}[1][\proofname]{\par
  \pushQED{\qed}%
  \normalfont \partopsep=\z@skip \topsep=\z@skip
  \trivlist
  \item[\hskip\labelsep
        \itshape
    #1\@addpunct{.}]\ignorespaces
}{%
  \popQED\endtrivlist\@endpefalse\bigskip
}
\makeatother

\usepackage{color}
\definecolor{lightgray}{rgb}{0.8, 0.8, 0.8}
\definecolor{darkgray}{rgb}{0.65, 0.65, 0.65}

\usepackage[bookmarks]{hyperref}
\hypersetup{
        colorlinks=true,
        linkcolor=black,
        anchorcolor=black,
        citecolor=black,
        urlcolor=black,
        pdfpagemode=UseThumbs,
        pdftitle={Letter graphs and geometric grid classes of permutations},
        pdfsubject={Combinatorics},
        pdfauthor={Alecu, Ferguson, Kante, Lozin, Vatter, and Zamaraev},
}

\newcounter{todocounter}

\theoremstyle{plain}
\newtheorem{theorem}{Theorem}[section]
\newtheorem{observation}[theorem]{Observation}
\newtheorem{proposition}[theorem]{Proposition}

\newtheorem{corollary}[theorem]{Corollary}

\theoremstyle{definition}

\newtheorem*{remark*}{Remark}

\makeatletter
\newfont{\footsc}{cmcsc10 at 8truept}
\newfont{\footbf}{cmbx10 at 8truept}
\newfont{\footrm}{cmr10 at 10truept}
\renewenvironment{abstract}%
                {
                  \begin{list}{}%
                     {\setlength{\rightmargin}{1in}%
                      \setlength{\leftmargin}{1in}}%
                   \item[]\ignorespaces\begin{small}}%
                 {\end{small}\unskip\end{list}}

\newcommand{\C}{\mathcal{C}}

\newcommand{\st}{\::\:}

\newcommand{\fnmatrix}[1]{\Big(\raisebox{0.5pt}{\scalebox{0.75}{\text{$\begin{matrix*}[r]#1\end{matrix*}$}}}\Big)}
\newcommand{\tinymatrix}[1]{\big(\raisebox{0.75pt}{\scalebox{0.5}{\text{$\begin{matrix*}[r]#1\end{matrix*}$}}}\big)}
\newcommand{\tinymatrixbigp}[1]{\left(\raisebox{0.75pt}{\scalebox{0.5}{\text{$\begin{matrix*}[r]#1\end{matrix*}$}}}\right)}

\newcommand{\Grid}{\operatorname{Grid}}
\newcommand{\Geom}{\operatorname{Geom}}
\newcommand{\zpm}{0/\mathord{\pm} 1}

\newcommand{\Xfig}{\begin{tikzpicture}[scale=1, anchor=base]
	\pgftransformxscale{0.112};
	\pgftransformyscale{0.112};
	\draw [semithick, line cap=round] (0,2)--(2,0) (0,0)--(2,2);
\end{tikzpicture}}
\newcommand{\Vfig}{\begin{tikzpicture}[scale=1, anchor=base]
	\pgftransformxscale{0.112};
	\pgftransformyscale{0.112};
	\draw [semithick, line cap=round] (0,2)--(1,0) (1,0)--(2,2);
\end{tikzpicture}}
\newcommand{\Vfigt}{\begin{tikzpicture}[scale=1, anchor=base]
	\pgftransformxscale{0.112};
	\pgftransformyscale{0.112};
	\draw [semithick, line cap=round] (0,0)--(1,2) (1,2)--(2,0);
\end{tikzpicture}}
\newcommand{\Vfigl}{\begin{tikzpicture}[scale=1, anchor=base]
	\pgftransformxscale{0.112};
	\pgftransformyscale{0.112};
	\draw [semithick, line cap=round] (2,0)--(0,1) (0,1)--(2,2);
\end{tikzpicture}}
\newcommand{\Vfigr}{\begin{tikzpicture}[scale=1, anchor=base]
	\pgftransformxscale{0.112};
	\pgftransformyscale{0.112};
	\draw [semithick, line cap=round] (0,0)--(2,1) (2,1)--(0,2);
\end{tikzpicture}}

\usepackage{xspace}

\renewcommand{\th}{\ensuremath{^{\text{\scriptsize th}}}\xspace}

\newcommand\mybullet{\raisebox{-5pt}{\normalsize \ensuremath{\bullet}}}

\makeatletter
\def\absdot{\@ifnextchar[{\@absdotlabel}{\@absdotnolabel}}
	\def\@absdotlabel[#1]#2{%
		\node at #2 {\normalsize \mybullet};
		\node at #2 [below=2pt] {\ensuremath{#1}};
	}
	\def\@absdotnolabel#1{%
		\node at #1 {\normalsize \mybullet};
	}
\newcommand{\plotpartialperm}[1]{
	\foreach \i/\j in {#1} {
		\absdot{(\i,\j)};
	};
}
\newcommand{\arc}[2]{
	\draw[thick] (#1,0) arc (180:0:{(#2-#1)/2});
	\absdot{(#1,0)};
	\absdot{(#2,0)};
}
\newcommand{\matching}[1]{
	\foreach \i/\j in {#1} {
		\arc{\i}{\j};
	};
}
\newcommand{\plotperm}[1]{
	\foreach \j [count=\i] in {#1} {
		\absdot{(\i,\j)};
	};
}
\newcommand{\plotpermbox}[4]{
	\draw [darkgray, thick, rounded corners=0.01, line cap=round]
		({#1-0.5}, {#2-0.5}) rectangle ({#3+0.5}, {#4+0.5});
}

\makeatletter
\let\start@align@nopar\start@align
\let\start@gather@nopar\start@gather
\let\start@multline@nopar\start@multline
\long\def\start@align{\par\start@align@nopar}
\long\def\start@gather{\par\start@gather@nopar}
\long\def\start@multline{\par\start@multline@nopar}
\makeatother

\newpagestyle{main}[\small]{
        \headrule
        \sethead[\usepage][][]
        {\sc Letter Graphs and Geometric Grid Classes}{}{\usepage}}

\title{\sc Letter Graphs and Geometric Grid Classes of Permutations}

\newcommand{\addressline}[1]{%
	\begin{minipage}{1.75in}
	\begin{center}
	\begin{footnotesize}
		\begin{tabular}{c}
		#1
		\end{tabular}
	\end{footnotesize}	
	\end{center}
	\end{minipage}
}

\author{
	\centering
	\begin{tabular}{ccc}
	Bogdan Alecu
		\footnote{Alecu's research was supported by EPSRC via a Doctoral Training Partnership grant to the University of Warwick.}
	&
	Robert Ferguson
	&
	Mamadou Moustapha Kant\'e%
		\footnote{Kant\'e's research was supported by the French Agency for Research under the ASSK project (ANR-18-CE40-0025).}
	\\
	\addressline{Mathematics Institute\\ University of Warwick\\ Coventry, UK}
	&
	\addressline{Department of Mathematics\\ University of Florida\\ Gainesville, Florida USA}
	&
	\addressline{Universit\'e Clermont Auvergne\\ LIMOS, CNRS\\ Aubi\`ere, France}
	\\[0.5in]
	Vadim Lozin
	&
	Vincent Vatter%
		\footnote{Vatter's research was supported by the Simons Foundation via award number 636113.}\,
	&
	Viktor Zamaraev
	\\
	\addressline{Mathematics Institute\\ University of Warwick\\ Coventry, UK}
	&
	\addressline{Department of Mathematics\\ University of Florida\\ Gainesville, Florida USA}
	&
	\addressline{Department of Computer Science\\ University of Liverpool\\ Liverpool, UK}
	\end{tabular}
}

\titleformat{\section}
        {\large\sc}
        {\thesection.}{1em}{}

\begin{document}
\maketitle

\renewcommand*{\thefootnote}{\fnsymbol{footnote}}
\addtocounter{footnote}{2}

\pagestyle{main}

\begin{abstract}
We uncover a connection between two seemingly unrelated notions: lettericity, from structural graph theory, and geometric griddability, from the world of permutation patterns. Both of these notions capture important structural properties of their respective classes of objects. We prove that these notions are equivalent in the sense that a permutation class is geometrically griddable if and only if the corresponding class of inversion graphs has bounded lettericity.
\end{abstract}

\section{Introduction}

Structural graph theory and the study of permutation patterns are two flourishing areas of mathematics. They share a few conceptual similarities---among others, the widespread featuring of forbidden substructures, and the prominent usage of decompositions. Nevertheless, the fields have developed their own separate terminologies, techniques, and goals.

Because of this, it happens periodically that equivalent notions appear in the two areas, independently of each other, and under different names. In most cases, a relationship between such notions can be derived directly from their definitions. Sometimes, however, this link can be much more subtle. This is indeed the case with the two notions studied in the present paper: lettericity on the one hand, and geometric griddability on the other. At a first reading of the definitions, nothing suggests that there should be a connection between these two notions. However, such a connection does exist, and was first observed by Alecu, Lozin, de Werra, and Zamaraev~\cite{alecu:letter-graphs-a:abstract,alecu:letter-graphs-a:}. They showed that if $\C$ is a geometrically griddable permutation class, then $G_\C$, the associated class of inversion graphs, has bounded lettericity. They conjectured that the converse to this statement also holds. We prove this conjecture, yielding the following result.

\begin{theorem}
\label{thm-ggc-lettericity}
The permutation class $\C$ is geometrically griddable if and only if the corresponding graph class $G_{\C}$ has bounded lettericity.
\end{theorem}

Despite their seemingly unrelated definitions, this result reflects the fact that the concepts of lettericity and geometric griddability capture the same structural data of their respective combinatorial objects: a partition of the elements and a linear ordering of those elements, encoded in words, that interacts nicely with the partition. Indeed, this close connection is evidenced by the way the word encodings imply well-quasi-orderability of the relevant classes (via Higman's lemma~\cite{higman:ordering-by-div:}; in fact both are \emph{labelled} well-quasi-ordered---see Atminas and Lozin~\cite{atminas:labelled-induce:} and Brignall and Vatter~\cite{brignall:labelled-well-q:}). Further context for Theorem~\ref{thm-ggc-lettericity} can be found in the article~\cite[Section~2]{lozin:from-words-to-g:} accompanying Lozin's plenary lecture at the conference \emph{LATA 2019}.

We define the terms involved in Sections~\ref{sec-defs-graphs}, \ref{sec-defs-perms}, and \ref{sec-ggc}. Then, in Section~\ref{sec-ggc-encoding}, we develop an alternative approach to geometric grid classes that is necessary for our work. The proof of our main result is contained in Section~\ref{sec-main-proof}, and we conclude in Section~\ref{sec-conclusion}. The unification of lettericity and geometric griddability established in this paper marks the starting point of a promising research direction, which we discuss in the conclusion.

\section{Graphs, Graph Classes, and Lettericity}
\label{sec-defs-graphs}

Our graphs are all finite, simple, and undirected. Given a graph $G$, we denote by $V(G)$ its vertex set and by $E(G)$ its edge set. We write $u\sim v$ to denote that the vertices $u$ and $v$ are adjacent, meaning that $uv\in E(G)$. We denote the complement of $G$ by $\overline{G}$. We denote the \emph{complete graph} or \emph{clique} on $n$ vertices by $K_n$, so an \emph{independent set} or \emph{co-clique} on $n$ vertices is denoted by $\overline{K}_n$. Given graphs $G$ and $H$ on disjoint vertex sets, we denote their \emph{disjoint union} by $G\uplus H$. This is the graph defined by $V(G\uplus H)=V(G)\cup V(H)$ and $E(G\uplus H)=E(G)\cup E(H)$. Given a graph $G$ and a natural number~$m$, we denote by~$mG$ the disjoint union of~$m$ copies of $G$ (that is, the copies are chosen to have disjoint vertex sets).

An \emph{induced subgraph} of the graph $G$ consists of a subset of the vertices of $G$ together with all edges of $G$ connecting vertices of this subset. This is the only notion of subgraph considered in this paper, and it is a partial order on the set of finite graphs. A \emph{class} of graphs is a set of finite graphs that is closed under isomorphism and closed downward under the induced subgraph ordering (these are also frequently called \emph{hereditary properties}). Thus if $\C$ is a graph class, $G\in\C$, and $H$ is isomorphic to an induced subgraph of $G$, then $H\in\C$. We introduce here the notions and results required, referring to the encyclopedic text of Brandst\"adt, Le, and Spinrad~\cite{brandstadt:graph-classes:-:} for further information and context on graph classes.

Letter graphs were introduced in 2002 by Petkov{\v{s}}ek~\cite{petkovsek:letter-graphs-a:}, and describe a way to build graphs from words. Let $\Sigma$ be a finite alphabet and let $D\subseteq\Sigma^2$ be a set of ordered pairs that we call a \emph{decoder}. Equivalently, we may think of the decoder as a digraph with vertex set $\Sigma$. For any word ${w=w(1)w(2)\cdots w(n)\in\Sigma^\ast}$, the \emph{letter graph of $w$ with respect to $D$} is defined to be the graph $\Gamma_D(w)$ defined by
\begin{eqnarray*}
	V(\Gamma_D(w))	&=&	\{1,2,\dots,n\}, \text{and}\\
	E(\Gamma_D(w))	&=&	\{ij\st\text{$i<j$ and $(w(i),w(j))\in D$}\}.	
\end{eqnarray*}
We frequently say that the vertex $i\in V(\Gamma_D(w))$ is \emph{encoded} by the letter $w(i)\in\Sigma$. If $\Sigma$ is an alphabet of size $k$, then the graph $\Gamma_D(w)$ is said to be a \emph{$k$-letter graph}.

It is easy to see that every graph on $n$ vertices is isomorphic to an $n$-letter graph---fix a labeling of the vertices of $G$ by $\Sigma=\{1,2,\dots,n\}$, let $w=12\cdots n$, and set $D=\{(i,j)\st\text{$i<j$ and $i\sim j$ in $G$}\}$. A triple $(\Sigma, D, w)$ such that $G$ is isomorphic to $\Gamma_D(w)$ is called a \emph{letter graph realization}, or simply \emph{lettering} of $G$. The minimum value of $k$ such that $G$ admits a lettering with $|\Sigma| = k$ is called the \emph{lettericity} of the graph $G$. Another fact that is easy to see is that if $G$ is isomorphic to the letter graph $\Gamma_D(w)$, then its complement $\overline{G}$ is isomorphic to $\Gamma_{\Sigma^2\setminus D}(w)$. From this, we immediately obtain the following.

\begin{proposition}[Petkov{\v{s}}ek~{\cite[Proposition~2 (iii)]{petkovsek:letter-graphs-a:}}]
\label{prop-lettericity-complement}
The lettericity of a graph and its complement are equal.
\end{proposition}

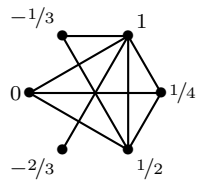
\begin{figure}
\begin{center}
\begin{footnotesize}
	\begin{tikzpicture}[scale=0.175, baseline=(current bounding box.center)]
	\draw [thick] (10,5)--(7.5,9.330127020);
	\draw [thick] (10,5)--(0,5);
	\draw [thick] (10,5)--(7.5,0.669872980);
	\draw [thick] (7.5,9.330127020)--(2.5,9.330127020);
	\draw [thick] (7.5,9.330127020)--(0,5);
	\draw [thick] (7.5,9.330127020)--(2.5,0.669872980);
	\draw [thick] (7.5,9.330127020)--(7.5,0.669872980);
	\draw [thick] (2.5,9.330127020)--(7.5,0.669872980);
	\draw [thick] (0,5)--(7.5,0.669872980);
	\plotpartialperm{10/5, 7.5/9.330127020, 2.5/9.330127020, 0/5, 2.5/0.669872980, 7.5/0.669872980};
	\node at (10,5) [right] {$\nicefrac{1}{4}$};
	\node at (0,5) [left] {$0$};
	\node at (2.5,9.330127020) [above left] {$-\nicefrac{1}{3}$};
	\node at (7.5,9.330127020) [above right] {$1$};
	\node at (2.5,0.669872980) [below left] {$-\nicefrac{2}{3}$};
	\node at (7.5,0.669872980) [below right] {$\nicefrac{1}{2}$};
	\end{tikzpicture}
\end{footnotesize}
\end{center}
\caption{A threshold graph with vertex weights indicated. Two vertices are connected by an edge if the sum of their weights is nonnegative.}
\label{fig-threshold-1}
\end{figure}

It is apparent from the definition that all of the vertices of $\Gamma_D(w)$ that are encoded by a single letter form either a clique or a co-clique---if the letter is $a\in\Sigma$, then these vertices form a clique if~$(a,a)\in D$ and they form a co-clique if $(a,a)\notin D$. This implies that the lettericity of a graph is bounded below by its \emph{co-chromatic number}, which is defined to be the minimum number of cliques and co-cliques that the vertices of the graph can be partitioned into. That said, the perfect matchings have co-chromatic number $2$ but unbounded lettericity, as we sketch below.

\begin{proposition}
\label{prop-lettericity-matching}
The set of perfect matchings $\{mK_2\st m\ge 1\}$ has unbounded lettericity.
\end{proposition}
\begin{proof}
Suppose to the contrary that the lettericity of every matching~$mK_2$ was bounded by some constant, say~$r$. In such an $r$-lettering of a matching, every edge has only $r^2$ possible labelings (or letterings). Therefore the matching $(mr^2+1)K_2$ must contain~$m$ edges with the same labels. It follows that the matching~$mK_2$ would therefore actually have lettericity $2$. However, it is not difficult to verify that $3K_2$ has lettericity $3$.
\end{proof}

In fact, without much additional effort, one can show that the lettericity of~$mK_2$ is precisely~$m$. This was claimed without proof by Petkov{\v{s}}ek~\cite[Section~5.3]{petkovsek:letter-graphs-a:}, and proved by Alecu, Lozin, and de Werra~\cite[Lemma~3]{alecu:the-micro-world:}; the approach used above is from Ferguson and Vatter~\cite{ferguson:letter-graphs-a:}.

%
%

We define the lettericity of a class of graphs to be the greatest lettericity of any of its members. Obviously this quantity is often infinite (for example, whenever a class contains all perfect matchings). The quintessential example of a graph class of finite lettericity is the class of threshold graphs introduced by Chv\'atal and Hammer, first in a 1973 research report and then in a 1977 paper~\cite{chvatal:aggregation-of-:}, and the subject of a book by Mahadev and Peled~\cite{mahadev:threshold-graph:}. The original definition of this class (which lends the class its name) is that the graph $G$ is a \emph{threshold graph} if there exists 
an assignment $v \mapsto c_v$ of real-valued weights to the vertices of $G$
such that $u\sim v$ if and only if $c_u+c_v\ge 0$ (in fact in the original definition this threshold value is arbitrary, but we lose nothing by taking it to be $0$). For example, the graph shown in Figure~\ref{fig-threshold-1} is a threshold graph, with weights as indicated by its labels.

The set of threshold graphs is a graph class, and it can be shown that it is precisely the class of graphs that do not contain induced subgraphs isomorphic to the matching on two edges $2K_2$, the cycle on four vertices $C_4=\overline{2K_2}$, or the path on four vertices $P_4$. More relevant to our concerns, it can also be shown that every threshold graph can be built from the empty graph by successively adding isolated vertices (adjacent to none of the previous vertices) and dominating vertices (adjacent to all of the previous vertices). Figure~\ref{fig-threshold-2} illustrates how our example from Figure~\ref{fig-threshold-1} can be drawn in this manner.

\begin{figure}
\begin{center}
\begin{footnotesize}
	\begin{tikzpicture}[scale=0.9, yscale=0.55, baseline=(current bounding box.center)]
	%
	\matching{1/4,2/4,6/1,6/2,6/3,6/4};
	\plotpartialperm{1/0,2/0,3/0,4/0,5/0,6/0};
	\draw [thick]
		(1,0)--(2,0)
		(3,0)--(4,0)
		(5,0)--(6,0);
	\node at (1,0) [above left] {$0$};
	\node at (2,0) [above left] {$\nicefrac{1}{4}$};
	\node at (3,0) [above] {$-\nicefrac{1}{3}$};
	\node at (4,0) [above right] {$\nicefrac{1}{2}$};
	\node at (5,0) [above] {$-\nicefrac{2}{3}$};
	\node at (6,0) [above] {$1$};
	\end{tikzpicture}
\quad\quad\quad
	\begin{tikzpicture}[scale=0.9, yscale=0.55, baseline=(current bounding box.center)]
	%
	\matching{1/4,2/4,6/1,6/2,6/3,6/4};
	\plotpartialperm{1/0,2/0,3/0,4/0,5/0,6/0};
	\draw [thick]
		(1,0)--(2,0)
		(3,0)--(4,0)
		(5,0)--(6,0);
	\node at (1,0) [above left] {$1$};
	\node at (2,0) [above left] {$2$};
	\node at (3,0) [above] {$3$};
	\node at (4,0) [above right] {$4$};
	\node at (5,0) [above] {$5$};
	\node at (6,0) [above] {$6$};
	\end{tikzpicture}
\end{footnotesize}
\end{center}
\caption{Two additional drawings of the threshold graph from Figure~\ref{fig-threshold-1}.}
\label{fig-threshold-2}
\end{figure}
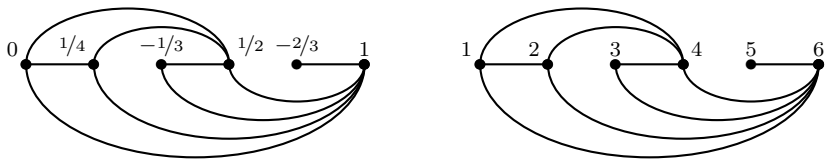

In the first of these drawings, we have retained the weights on the vertices for ease of comparison with our previous drawing of this graph, while in the second we have labelled vertices by their order of addition to the graph (as isolated or dominating vertices). From this perspective, it is clear how to encode threshold graphs as $2$-letter graphs; using the two-letter alphabet $\Sigma=\{\mathsf{i},\mathsf{d}\}$ and the decoder $D=\{(\mathsf{i},\mathsf{d}),(\mathsf{d},\mathsf{d})\}$, we encode vertices in their order of addition to the graph by $\mathsf{i}$ if they are added as isolated vertices or by $\mathsf{d}$ if they are added as dominating vertices. Thus with this alphabet and decoder, our example can be seen to be isomorphic to $\Gamma_D(\mathsf{ididid})$. In fact it is also isomorphic to $\Gamma_D(\mathsf{ddidid})$ because the first letter does not matter in this case, but that is beside the point.

The discussion above demonstrates that the class of threshold graphs has lettericity $2$. In this case, every threshold graph can be encoded using the same decoder, but note that our definition of the lettericity of a class does not require us to use the same decoder for every graph in the class. That said, since there are only finitely many possible decoders for a given finite alphabet, if a class has finite lettericity, then we can also find a single decoder $D$ over a larger but still finite alphabet $\Sigma$ such that all members of the class can be expressed as $\Gamma_D(w)$ for some word~$w\in\Sigma^\ast$.

\section{Permutations, Permutation Classes, and Grids}
\label{sec-defs-perms}

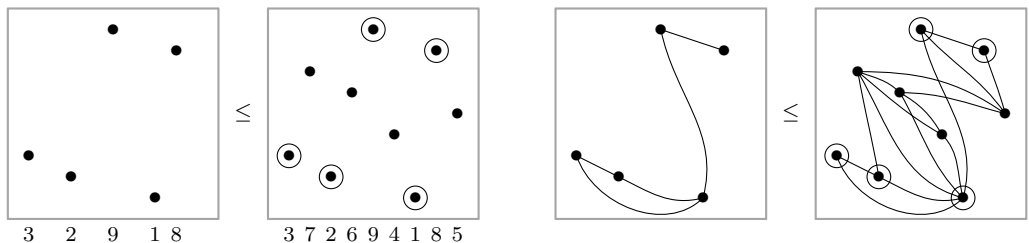
\begin{figure}
\begin{footnotesize}
\begin{center}
	\begin{tikzpicture}[scale={2.8/10}, baseline=(current bounding box.north)]
		\plotpermbox{0.5}{0.5}{9.5}{9.5};
		\plotpartialperm{1/3,3/2,5/9,7/1,8/8};
		\node at (1,0) [below] {$3$};
		\node at (3,0) [below] {$2$};
		\node at (5,0) [below] {$9$};
		\node at (7,0) [below] {$1$};
		\node at (8,0) [below] {$8$};
	\end{tikzpicture}
	\begin{tikzpicture}[scale={2.8/10}, baseline=(current bounding box.north)]
		\draw[color=white] (0,0)--(0,10);
		\node at (0,5) {$\le$};
		\node at (0,0) [below] {\phantom{$1$}};
	\end{tikzpicture}
	\begin{tikzpicture}[scale={2.8/10}, baseline=(current bounding box.north)]
		\plotpermbox{0.5}{0.5}{9.5}{9.5};
		\plotperm{3,7,2,6,9,4,1,8,5};
		\draw (1,3) circle (16pt);
		\draw (3,2) circle (16pt);
		\draw (5,9) circle (16pt);
		\draw (7,1) circle (16pt);
		\draw (8,8) circle (16pt);
		\node at (1,0) [below] {$3$};
		\node at (2,0) [below] {$7$};
		\node at (3,0) [below] {$2$};
		\node at (4,0) [below] {$6$};
		\node at (5,0) [below] {$9$};
		\node at (6,0) [below] {$4$};
		\node at (7,0) [below] {$1$};
		\node at (8,0) [below] {$8$};
		\node at (9,0) [below] {$5$};
	\end{tikzpicture}
\quad\quad\quad
	\begin{tikzpicture}[scale={2.8/10}, baseline=(current bounding box.north)]
	    \draw (1,3)--(3,2);
	    \draw (5,9) to [out=-78.4349488, in=75] (7,1);
	    \draw (3,2) to [out=-26.5650512, in=195] (7,1);
	    \draw (1,3) to [out=-71.5650512, in=225] (7,1);
	    \draw (5,9) to [out=-18.4349488, in=161.5650512] (8,8);
		\plotpermbox{0.5}{0.5}{9.5}{9.5};
		\plotpartialperm{1/3,3/2,5/9,7/1,8/8};
	\end{tikzpicture}
	\begin{tikzpicture}[scale={2.8/10}, baseline=(current bounding box.north)]
		\draw[color=white] (0,0)--(0,10);
		\node at (0,5) {$\le$};
		\node at (0,0) [below] {\phantom{$1$}};
	\end{tikzpicture}
	\begin{tikzpicture}[scale={2.8/10}, baseline=(current bounding box.north)]
	    \draw (2,7) to [out=-80, in=100] (3,2);
	    \draw (1,3)--(3,2);
	    \draw (2,7) to [out=-20] (4,6);
	    \draw (2,7) to [out=-40, in=150] (6,4);
	    \draw (4,6) to [out=-30, in=120] (6,4);
	    \draw (5,9) to [out=-78.4349488, in=75] (7,1);
	    \draw (6,4) to [out=-45, in=105] (7,1);
	    \draw (4,6) to [out=-60, in=135] (7,1);
	    \draw (2,7) to [out=-60, in=165] (7,1);
	    \draw (3,2) to [out=-26.5650512, in=195] (7,1);
	    \draw (1,3) to [out=-71.5650512, in=225] (7,1);
	    \draw (5,9) to [out=-18.4349488, in=161.5650512] (8,8);
	    \draw (8,8) to [out=-71.5650512, in=105] (9,5);
	    \draw (5,9) to [out=-48.4349488, in=125] (9,5);
	    \draw (2,7) to [out=0, in=145] (9,5);
	    \draw (4,6) to [out=0, in=165] (9,5);
		\plotpermbox{0.5}{0.5}{9.5}{9.5};
		\plotperm{3,7,2,6,9,4,1,8,5};
		\draw (1,3) circle (16pt);
		\draw (3,2) circle (16pt);
		\draw (5,9) circle (16pt);
		\draw (7,1) circle (16pt);
		\draw (8,8) circle (16pt);
	\end{tikzpicture}
\end{center}
\end{footnotesize}
\caption{The containment order on permutations and the induced subgraph order on their corresponding inversion graphs.}
\label{fig-perm-contain}
\end{figure}

We think of permutations in one-line notation, so a permutation of \emph{length} $n$ is simply an ordering of the set $\{1,2,\dots,n\}$. The permutation~$\pi$ \emph{contains} the permutation $\sigma$ of length $k$ if it has a subsequence of length $k$ that is \emph{order isomorphic} to $\sigma$, by which we mean that the subsequence has the same pairwise comparisons as $\sigma$. For example, $\pi=372694185$ contains $\sigma=32514$, as witnessed by its subsequence $32918$, but~$\pi$ avoids $54321$ because it has no decreasing subsequence of length five. This containment is depicted graphically on the left of Figure~\ref{fig-perm-contain}. As in this figure, we often identify a permutation~$\pi$ with its \emph{plot}: the set of points $\{(i,\pi(i))\}$ in the plane.

A \emph{class} of permutations is a set of permutations that is closed downward under the containment order defined above. Thus if $\C$ is a permutation class, $\pi\in\C$, and $\sigma$ is contained in~$\pi$, then $\sigma\in\C$. We introduce here the notions and results required, referring to the survey by Vatter~\cite{vatter:permutation-cla:} for further information and context on permutation classes.

The \emph{inversion graph} of the permutation $\pi=\pi(1)\cdots\pi(n)$ is the graph $G_\pi$ on the vertices $\{1,\dots,n\}$ in which $i\sim j$ if and only if $\pi(i)$ and $\pi(j)$ form an \emph{inversion}, meaning that either $i<j$ and $\pi(i)>\pi(j)$ or $i>j$ and $\pi(i)<\pi(j)$, or equivalently, if and only if
\[
	(i-j)(\pi(i)-\pi(j))<0.
\]
These graphs are commonly called permutation graphs in the literature, but we use the term inversion graph here to emphasize \emph{how} they arise from permutations. As shown on the right of Figure~\ref{fig-perm-contain}, to obtain the inversion graph of a permutation from its plot, we simply add all edges between pairs of entries in which one lies to the northwest of the other.

If the permutation $\sigma$ is contained in the permutation~$\pi$, then it is clear that $G_\sigma$ is an induced subgraph of $G_\pi$. The converse, however, does not generally hold; for one example, $G_{2413}\cong G_{3142}\cong P_4$, but neither $2413$ nor $3142$ is contained in the other. That said, only the increasing and decreasing permutations correspond to cliques and co-cliques, so every clique (resp., co-clique) in $G_\pi$ corresponds to a decreasing (resp., increasing) subsequence of~$\pi$.

The monotone grid class of a~$\zpm$~matrix consists of all those permutations whose plots can be sliced into rectangles in such a way that each of the resulting rectangular sections contains a monotone subsequence, as specified by the matrix. More precisely, suppose that~$M$ is a~$\zpm$~matrix of size~$t\times u$. In order for the entries of our matrices to align with the plots of permutations, we index matrices in cartesian coordinates, so this means for us that~$M$ consists of $t$ columns and $u$ rows. An \emph{$M$-gridding} of the permutation~$\pi$ of length $n$ is a choice of \emph{column divisions} $1=x_1\le\cdots\le x_{t+1}\le n+1$ and \emph{row divisions} $1=y_1\le\cdots\le y_{u+1}\le n+1$ such that for all $i$ and $j$, the subsequence of~$\pi$ with indices in $[x_i,x_{i+1})$ and values in $[y_j,y_{j+1})$ is increasing if~$M(i,j)=1$, decreasing if~$M(i,j)=-1$, or empty if~$M(i,j)=0$. The \emph{monotone grid class of~$M$} is the set of all permutations that possess an~$M$-gridding, and is denoted by $\Grid(M)$.

\begin{figure}
\begin{center}
	\begin{tikzpicture}[scale={2.8/7}, baseline=(current bounding box.center)]
		\plotpermbox{1}{1}{3}{4};
		\plotpermbox{1}{5}{3}{6};
		\plotpermbox{4}{1}{6}{4};
		\plotpermbox{4}{5}{6}{6};
		\plotperm{5,2,4,3,6,1};
	\end{tikzpicture}
\end{center}
\caption[A monotone gridding of a permutation.]{A $\tinymatrix{-1&1\\1&-1}$-gridding of the permutation $524361$ (recall that permutations are read from left to right). Here the column divisions are given by $x_1=1$, $x_2=4$, and $x_3=7$, while the row divisions are $y_1=1$, $y_2=5$, and $y_3=7$.}
\label{fig-mono-gridding-skew-merged}
\end{figure}
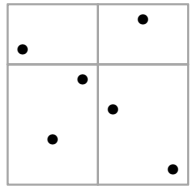


Figure~\ref{fig-mono-gridding-skew-merged} depicts a member of the grid class of $\tinymatrix{-1&1\\1&-1}$. This particular grid class admits a simpler definition: it consists of those permutations whose entries can be partitioned into an increasing subsequence and a decreasing subsequence. This is the definition under which these permutations were first studied, by Stankova~\cite{stankova:forbidden-subse:}, who called them \emph{skew-merged}. Stankova proved that a permutation is skew-merged if and only if it contains neither $2143$ nor $3412$.

The permutation class $\C$ is said to be \emph{monotonically griddable} if $\C\subseteq\Grid(M)$ for some~$\zpm$~matrix~$M$. Monotone grid classes were first considered in a 2003 paper by Murphy and Vatter~\cite{murphy:profile-classes:}, although they called them ``profile classes''. In their characterization of permutation classes of polynomial growth (see also~\cite{homberger:on-the-effectiv:}), Huczynska and Vatter~\cite{huczynska:grid-classes-an:} gave monotone grid classes their name and characterized the monotonically griddable permutation classes. Stankova's characterization of the skew-merged permutations in fact constitutes the base case of the inductive proof of the following result in~\cite{huczynska:grid-classes-an:}. For a generalization of this result we refer to Vatter~\cite[Theorem~3.1]{vatter:small-permutati:}.

\begin{theorem}[Huczynska and Vatter~{\cite[Theorem~2.5]{huczynska:grid-classes-an:}}]
\label{thm-mono-griddable}
The permutation class $\C$ is monotonically griddable if and only if there is a natural number~$m$ such that $\C$ contains neither $2143\cdots (2m)(2m-1)$ nor $(2m-1)(2m)\cdots 3412$.
\end{theorem}

The graphical analogues of the skew-merged permutations are the \emph{split graphs}: the graphs whose vertices can be partitioned into a clique and an co-clique. This graph class was first studied by F{\"{o}}ldes and Hammer~\cite{foldes:split-graphs:} in 1977, who showed that a graph is split if and only if it contains none of the graphs $2K_2$, $\overline{2K_2}=C_4$, or $C_5$ as induced subgraphs. We note that inversion graphs never contain induced cycles of length five or more, so a class of inversion graphs is split if and only if it contains neither $2K_2$ nor $\overline{2K_2}=C_4$.

More generally, we observe that
\begin{eqnarray*}
	G_{2143\cdots (2m)(2m-1)}	&=&	mK_2,	\text{and}\\
	G_{(2m-1)(2m)\cdots 3412}	&=&	\overline{mK_2}.
\end{eqnarray*}
Just as only the increasing and decreasing permutations correspond to cliques and co-cliques, only the permutation $2143\cdots (2m)(2m-1)$ corresponds to the matching~$mK_2$, and only its reverse ${(2m-1)(2m)\cdots 3412}$ corresponds to $\overline{mK_2}$. This gives us the following consequence of Theorem~\ref{thm-mono-griddable}, which is somewhat analogous to the main result of our paper, Theorem~\ref{thm-ggc-lettericity}.

\begin{corollary}
\label{cor-mono-griddable-graphs-rK2}
The permutation class $\C$ is monotonically griddable if and only if there is a natural number~$m$ such that the corresponding graph class $G_\C$ contains neither~$mK_2$ nor $\overline{mK_2}$.	
\end{corollary}

Proposition~\ref{prop-lettericity-matching} shows that the lettericity of the perfect matchings~$mK_2$ is unbounded, and thus by Proposition~\ref{prop-lettericity-complement}, the lettericity of their complements $\overline{mK_2}$ is also unbounded. Thus we have the following consequence of Corollary~\ref{cor-mono-griddable-graphs-rK2}.

\begin{corollary}
\label{cor-mono-griddable-graphs}
If the graph class $G_\C$ has bounded lettericity, then the permutation class $\C$ is monotonically griddable.
\end{corollary}

The converse of Corollary~\ref{cor-mono-griddable-graphs} does not hold, as there are monotonically griddable permutation classes for which the corresponding graph classes have unbounded lettericity (in fact we've just presented an example: the skew-merged permutations are monotonically griddable, but the corresponding graph class, the split graphs, has unbounded lettericity). We close this section by remarking that the graphical analogues of general monotonically griddable permutation classes have only recently attracted attention, in the work of Atminas~\cite{atminas:classes-of-grap:}.

\section{Geometric Grid Classes}
\label{sec-ggc}

From our previous discussion, it follows that Chv\'atal and Hammer's threshold graphs are a subclass of F{\"{o}}ldes and Hammer's split graphs. This can be seen either by recalling the iterative construction of threshold graphs (by adding isolated or dominating vertices in sequence), or from the forbidden subgraph characterizations of the two classes (graphs that avoid $2K_2$, $\overline{2K_2}=C_4$, and $P_4$ will necessarily avoid $2K_2$, $\overline{2K_2}=C_4$, and $C_5$). Moreover, this containment of graph classes is strict because $P_4$ is a split graph, but is not a threshold graph. The relationship between the class of split graphs and the class of threshold graphs is a graphical analogue of the relationship between monotone grid classes and geometric grid classes. Having defined the former in the previous section, we define the latter here.

Geometric grid classes may be defined in terms of drawings of permutations in the plane $\mathbb{R}^2$, which we think of as divided into \emph{cells} whose corners are given by the integer lattice $\mathbb{Z}^2$. The \emph{standard figure} of a~$\zpm$~matrix~$M$ is the point set in $\mathbb{R}^2$ consisting of
\begin{itemize}[topsep=-4pt, itemsep=-2pt]
\item the increasing open line segment from $(i-1,j-1)$ to $(i,j)$ if~$M(i,j)=1$ or
\item the decreasing open line segment from $(i-1,j)$ to $(i,j-1)$ if~$M(i,j)=-1$.
\end{itemize}
An example is shown on the left of Figure~\ref{fig-ggc-example}; we remind the reader again that we index matrices in cartesian coordinates. We further extend this indexing to the cells of the standard figure of~$M$, so that the cell $(2,1)$ of this standard figure corresponds to the entry~$M(2,1)$ and lies in the second column and first row.

\begin{figure}
\begin{footnotesize}
\begin{center}
	\begin{tikzpicture}[scale=1.4, baseline=(current bounding box.center)]
		\draw[step=1, darkgray, thick, line cap=round] (0,0) grid (3,2);
		\draw [very thick, line cap=round] (0,2)--(1,1);
		\draw [very thick, line cap=round] (2,0)--(1,1);
		\draw [very thick, line cap=round] (2,2)--(1,1);
		\draw [very thick, line cap=round] (3,2)--(2,1);
		\draw [very thick, line cap=round] (3,0)--(2,1);
	\end{tikzpicture}
\quad\quad\quad
	\begin{tikzpicture}[scale=1.4, baseline=(current bounding box.center)]
		\draw[step=1, darkgray, thick, line cap=round] (0,0) grid (3,2);
		\draw [very thick, line cap=round] (0,2)--(1,1);
		\draw [very thick, line cap=round] (2,0)--(1,1);
		\draw [very thick, line cap=round] (2,2)--(1,1);
		\draw [very thick, line cap=round] (3,2)--(2,1);
		\draw [very thick, line cap=round] (3,0)--(2,1);
		\plotpartialperm{1.8/1.8, 0.3/1.7, 2.6/0.4, 2.5/1.5, 2.4/0.6, 1.3/0.7, 0.8/1.2};
		\node at (1.8,1.8) [below right] {$7$};
		\node at (0.3,1.7) [above right] {$6$};
		\node at (2.6,0.4) [below left] {$1$};
		\node at (2.5,1.5) [above left] {$5$};
		\node at (2.4,0.6) [below left] {$2$};
		\node at (1.3,0.7) [above right] {$3$};
		\node at (0.8,1.2) [above right] {$4$};
	\end{tikzpicture}
\end{center}
\end{footnotesize}
\caption[An example of a standard figure and a permutation drawn on that standard figure.]{On the left, the standard figure of the matrix $\tinymatrix{-1&1&1\\0&-1&-1}$. On the right, a drawing of the permutation $6437251$ on this figure.}
\label{fig-ggc-example}
\end{figure}
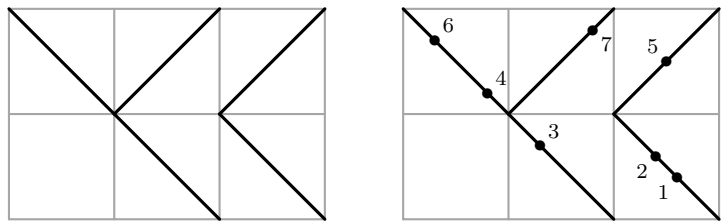

Any subset of $n$ points of the standard figure of~$M$ in which no two points lie on a common horizontal or vertical line (such point sets are often called \emph{generic}) corresponds to a permutation by labeling the points from $1$ to $n$ from bottom to top and then reading their labels from left to right. We say that the permutation~$\pi$ of length $n$ can be \emph{drawn} on the standard figure of~$M$ if it corresponds to such a subset. An example is shown on the right of Figure~\ref{fig-ggc-example}.

The \emph{geometric grid class} of~$M$ is the class of all permutations that can be drawn on the standard figure of~$M$ in this manner. We denote this class by $\Geom(M)$. The permutation class $\C$ is called \emph{geometrically griddable} if it is contained in a geometric grid class, that is, if $\C\subseteq\Geom(M)$ for some (finite)~$\zpm$~matrix~$M$.

It is clear that
\[
	\Geom(M)\subseteq\Grid(M)
\]
for all~$\zpm$ matrices~$M$. In some cases (characterized by whether a certain graph determined by the matrix~$M$ is a forest---see Vatter~\cite[Proposition~12.4.12]{vatter:permutation-cla:}), this is an equality, but not always. In this work, we use only the fact that geometrically griddable classes are necessarily monotonically griddable. Thus in particular, Theorem~\ref{thm-mono-griddable} implies that every geometrically griddable permutation class must avoid $2143\cdots (2m)(2m-1)$ and $(2m-1)(2m)\cdots 3412$ for some natural number~$M$. We note that there is, to-date, no characterization such as Theorem~\ref{thm-mono-griddable} for geometrically griddable classes themselves.

One of the first examples of a geometric grid class to be considered (in fact, before the machinery of geometric grid classes had been established) is the class of permutations that can be \emph{drawn on a circle}~\cite{vatter:on-points-drawn:}. This class is the same as the geometric grid class of the matrix $\tinymatrix{1&-1\\-1&1}$.

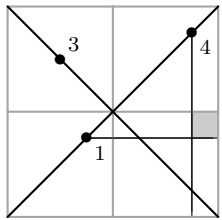
\begin{figure}
\begin{footnotesize}
\begin{center}
	\begin{tikzpicture}[scale=1.4, baseline=(current bounding box.center)]
		\draw [lightgray, fill=lightgray] (1.75,0.75) rectangle (2,1);
		\draw (0.75,0.75)--(2,0.75);
		\draw (1.75,1.75)--(1.75,0);
		\draw[step=1, darkgray, thick, line cap=round] (0,0) grid (2,2);
		\draw (0.75,0.75)--(1.95,0.75);
		\draw (1.75,1.75)--(1.75,0.05);
		\plotpartialperm{0.5/1.5, 0.75/0.75, 1.75/1.75};
		\node at (0.5,1.5) [above right] {$3$};
		\node at (0.75,0.75) [below right] {$1$};
		\node at (1.75,1.75) [below right] {$4$};
		\draw [thick, line cap=round] (0,2)--(2,0);
		\draw [thick, line cap=round] (0,0)--(2,2);
	\end{tikzpicture}
\end{center}
\end{footnotesize}
\caption
	[The permutation $3142$ cannot be drawn on an $\mathsf{X}$. Once we place the $3$, $1$, and $4$ on the $\mathsf{X}$, there is no place for the $2$ to lie simultaneously above the $1$ and to the right of the $4$.]
	{The permutation $3142$ cannot be draw on an $\Xfig$. Once we place the $3$, $1$, and $4$ on the $\Xfig$, there is no place for the $2$ to lie simultaneously above the $1$ and to the right of the $4$.}
\label{fig-no-3142-on-X}
\end{figure}

More relevant to our previous discussion is the class of permutations that can be drawn on an $\Xfig$, or in other words, the geometric grid class of the matrix $\tinymatrix{-1&1\\1&-1}$. As with the permutations that can be drawn on a circle above, this geometric grid class is a strict subset of the corresponding monotone grid class:
\[
	\Geom\fnmatrix{-1&1\\1&-1}\subsetneq\Grid\fnmatrix{-1&1\\1&-1}.
\]
An example of why this inclusion is strict is shown in Figure~\ref{fig-no-3142-on-X}. In fact, this geometric grid class consists of those skew-merged permutations that are also \emph{separable}, meaning that they avoid $2413$ and $3142$.

If $G_\pi$ is a threshold graph, then~$\pi$ can be drawn on an $\Xfig$. Indeed, in this case we can go a step further: if $G$ is a threshold graph, then $G\cong G_\pi$ for a permutation~$\pi$ that can be drawn on an $\Xfig$. But we do not need to consider all permutations that can be drawn on an $\Xfig$ to get the class of threshold graphs. In~\cite{golumbic:threshold-graph:}%
\footnote{This theorem can also be found in Golumbic's textbook \emph{Algorithmic Graph Theory and Perfect Graphs}~\cite{golumbic:algorithmic-gra:}.},
Golumbic showed that the threshold graphs can be obtained as inversion graphs of members of the class
\[
	\Geom\fnmatrix{-1\\1},
\]
which is the class of permutations that can be drawn on a $\Vfigr$. Golumbic's characterization is in terms of the \emph{shuffle product}, where for words $u$ and $v$, the set $u\shuffle v$ consists of all words formed by interleaving the letters of $u$ and $v$ in all possible ways. For example,
\newcommand{\snum}[1]{\textcolor[rgb]{0, 0.1294117647, 0.64705882352}{#1}}
\newcommand{\lnum}[1]{\textcolor[rgb]{0.98039215686, 0.27450980392, 0.0862745098}{#1}}
\[
	\snum{12}\shuffle \lnum{543}
	=
	\{	\snum{12}\lnum{543},
		\snum{1}\lnum{5}\snum{2}\lnum{43},
		\snum{1}\lnum{54}\snum{2}\lnum{3},
		\snum{1}\lnum{543}\snum{2},
		\lnum{5}\snum{12}\lnum{43},
		\lnum{5}\snum{1}\lnum{4}\snum{2}\lnum{3},
		\lnum{5}\snum{1}\lnum{43}\snum{2},
		\lnum{54}\snum{12}\lnum{3},
		\lnum{54}\snum{1}\lnum{3}\snum{2},
		\lnum{543}\snum{12}
	\}.	
\]
His result then states that the threshold graphs are precisely the inversion graphs of the permutations contained in the shuffle product
\[
	\Bigl(12\cdots m\Bigr)\shuffle \Bigl(n(n-1)\cdots (m+1)\Bigr)
\]
for some nonnegative integers $n$ and~$m$, which is equivalent to being drawn on a $\Vfigr$. By symmetry, one can also obtain the threshold graphs as the inversion graphs of permutations that can be drawn on a $\Vfig$, a $\Vfigl$, or a $\Vfigt$; see Figure~\ref{fig-five-ididid}.

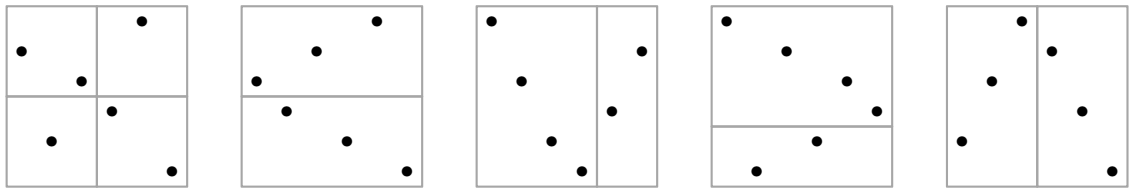
\begin{figure}
\begin{footnotesize}
\begin{center}
	\begin{tikzpicture}[scale={2.8/7}, baseline=(current bounding box.center)]
		\plotpermbox{1}{1}{3}{3};
		\plotpermbox{1}{4}{3}{6};
		\plotpermbox{4}{1}{6}{3};
		\plotpermbox{4}{4}{6}{6};
		\plotperm{5,2,4,3,6,1};
	\end{tikzpicture}
\quad\quad
	\begin{tikzpicture}[scale={2.8/7}, baseline=(current bounding box.center)]
		\plotpermbox{1}{1}{6}{3};
		\plotpermbox{1}{4}{6}{6};
		\plotperm{4,3,5,2,6,1};
	\end{tikzpicture}
\quad\quad
	\begin{tikzpicture}[scale={2.8/7}, baseline=(current bounding box.center)]
		\plotpermbox{1}{1}{4}{6};
		\plotpermbox{5}{1}{6}{6};
		\plotperm{6,4,2,1,3,5};
	\end{tikzpicture}
\quad\quad
	\begin{tikzpicture}[scale={2.8/7}, baseline=(current bounding box.center)]
		\plotpermbox{1}{1}{6}{2};
		\plotpermbox{1}{3}{6}{6};
		\plotperm{6,1,5,2,4,3};
	\end{tikzpicture}
\quad\quad
	\begin{tikzpicture}[scale={2.8/7}, baseline=(current bounding box.center)]
		\plotpermbox{1}{1}{3}{6};
		\plotpermbox{4}{1}{6}{6};
		\plotperm{2,4,6,5,3,1};
	\end{tikzpicture}
\end{center}
\end{footnotesize}
\caption{The inversion graphs of all five of these permutations are isomorphic to the threshold graph depicted in Figures~\ref{fig-threshold-1} and \ref{fig-threshold-2}.}
\label{fig-five-ididid}
\end{figure}

This example illustrates one difficulty in the proof of Theorem~\ref{thm-ggc-lettericity}. Letting $\mathcal{X}$ denote the class of permutations that can be drawn on an $\Xfig$ and $\mathcal{V}$ denote the class of permutations that can be drawn on a $\Vfig$, we see that $\mathcal{V}$ is a proper subclass of $\mathcal{X}$, but that these two classes correspond to the same class of inversion graphs (the threshold graphs). In symbols,
\[
	\mathcal{V}\subsetneq \mathcal{X}\quad\text{but}\quad G_{\mathcal{V}}=G_{\mathcal{X}}.
\]
In particular, this example shows that we may need more cells in our geometric grid classes than letters in the corresponding class of letter graphs---the threshold graphs have lettericity $2$ while the permutations that can be drawn on an $\Xfig$ require $4$ nonempty cells. In our proof, we make no effort to conserve cells, and consequently this explosion will be quite dramatic.

\section{Encoding Geometric Grid Classes}
\label{sec-ggc-encoding}

In the previous section we defined geometric grid classes geometrically, as classes of permutations that can be drawn on the standard figures of their defining matrices. While that description suffices as a definition, it is not very convenient for actually proving theorems about geometric grid classes. Here we give an order-theoretic viewpoint of geometric griddability, and show how one direction of Theorem~\ref{thm-ggc-lettericity} follows readily. We note that this approach essentially follows Vatter and Waton~\cite{vatter:on-partial-well:}, and thus actually pre-dates the definition of geometric grid classes.

The matrix~$M$ of size~$t\times u$ is a \emph{partial multiplication matrix} if there are \emph{column} and \emph{row signs}
\[
	c_1,\dots,c_t,r_1,\dots,r_u\in\{1,-1\}
\]
such that every nonzero entry~$M(k,\ell)$ is equal to $c_kr_\ell$. Partial multiplication matrices are necessarily~$\zpm$ matrices. However, not every~$\zpm$~matrix is a partial multiplication matrix, as can be seen by considering the matrix $\tinymatrix{1&-1\\1&1}$. It is nevertheless true that every geometric grid class can be viewed as the geometric grid class of a partial multiplication matrix. To see this, let~$M$ be an arbitrary~$\zpm$~matrix, and form a new matrix~$M^{\times 2}$ by replacing each entry of~$M$ by an appropriate $2\times 2$ block, according to the substitution rules
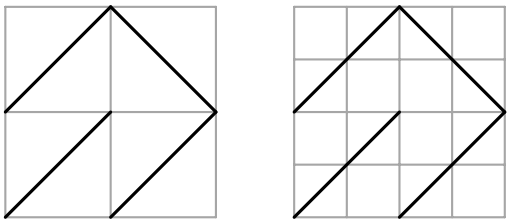
\begin{figure}
\begin{footnotesize}
\begin{center}
	\begin{tikzpicture}[scale=1.4, baseline=(current bounding box.center)]
		\draw[step=1, darkgray, thick, line cap=round] (0,0) grid (2,2);
		\draw [very thick, line cap=round] (0,1)--(1,2);
		\draw [very thick, line cap=round] (0,0)--(1,1);
		\draw [very thick, line cap=round] (1,2)--(2,1);
		\draw [very thick, line cap=round] (1,0)--(2,1);
	\end{tikzpicture}
\quad\quad\quad
	\begin{tikzpicture}[scale=1.4, baseline=(current bounding box.center)]
		\draw[step=0.5, darkgray, thick, line cap=round] (0,0) grid (2,2);
		\draw [very thick, line cap=round] (0,1)--(1,2);
		\draw [very thick, line cap=round] (0,0)--(1,1);
		\draw [very thick, line cap=round] (1,2)--(2,1);
		\draw [very thick, line cap=round] (1,0)--(2,1);
	\end{tikzpicture}
\end{center}
\end{footnotesize}
\caption[The standard figure of a matrix and of the equivalent partial multiplication matrix.]{The standard figure of the matrix~$M=\tinymatrix{1&-1\\1&1}$ on the left and of~$M^{\times 2}=\tinymatrixbigp{0&1&-1&0\\1&0&0&-1\\0&1&0&1\\1&0&1&0}$ on the right.}
\label{fig-ggc-pmm-example}
\end{figure}
\[
	0\leftarrow\fnmatrix{0&0\\0&0},
	\quad
	1\leftarrow\fnmatrix{0&1\\1&0},
	\quad\text{and}\quad
	-1\leftarrow\fnmatrix{-1&0\\0&-1}.
\]
An example of this process is shown in Figure~\ref{fig-ggc-pmm-example}. As indicated by that figure, the standard figure of~$M^{\times 2}$ is simply a scaled copy of the standard figure of~$M$, and thus
\[
	\Geom(M^{\times 2})=\Geom(M).
%
%
\]
By definition, $M^{\times 2}$ is comprised of $2\times 2$ blocks, each obtained by replacing two or four entries of the matrix $\tinymatrix{-1&1\\1&-1}$ by $0$. As the matrix $\tinymatrix{-1&1\\1&-1}$ is a partial multiplication matrix, it follows that~$M^{\times 2}$ is a partial multiplication matrix. In fact, remembering that we index matrices in cartesian coordinates, we can use the same column and row signs, $c_k=(-1)^k$ and $r_\ell=(-1)^\ell$, for every matrix of the form~$M^{\times 2}$ (although we do not assume this). We may therefore restrict our attention to partial multiplication matrices, as recorded below.

\begin{proposition}[Albert, Atkinson, Bouvel, Ru\v{s}kuc, and Vatter~{\cite[Proposition~4.2]{albert:geometric-grid-:}}]
\label{prop-geom-pmm}
Every geometric grid class is the geometric grid class of a partial multiplication matrix.
\end{proposition}


With this result in hand, suppose that~$M$ is a partial multiplication matrix with column and row signs~$(c_k)$ and~$(r_\ell)$, respectively, and that $\pi\in\Grid(M)$. We fix a particular $M$-gridding of~$\pi$ and denote this \emph{gridded permutation} (the permutation~$\pi$ together with an $M$-gridding) by $\pi^\gridded$. While $\pi^\gridded$ is technically a permutation together with a choice of column and row divisions, in practice we often think of it as simply a permutation together with an assignment of its entries into cells. We now use the column and row signs to orient the entries of $\pi^\gridded$ in two different manners, one geometric and the other order-theoretic. We begin with the geometric interpretation.

We interpret the column sign~$c_k$ as specifying whether the entries of $\pi^\gridded$ in column~$k$ are oriented from left to right (if~$c_k=1$) or from right to left (if~$c_k=-1$). We similarly interpret the row sign $r_\ell$ as specifying whether the entries of $\pi^\gridded$ in row~$\ell$ are oriented from bottom to top (if $r_\ell=1$) or from top to bottom (if $r_\ell=-1$). We depict these orientations with arrows to the left of and below the corresponding rows and columns of the standard figure of the matrix~$M$, as shown in Figure~\ref{fig-ggc-example-with-signs}.

These orientations are consistent with the contents of each cell in the sense that $M(k,\ell)=c_k r_\ell$ if it is nonzero, because $M$ is a partial multiplication matrix. The column and row signs  allow us to further define a \emph{base point} in each nonempty cell of the standard figure, which we take to be the point from which all the arrows of that cell originate, and we further depict the line segment in each cell (upon which the points lie) as a ray pointing away from the base point.

This is the geometric orientation of the entries of $\pi^\gridded$. In the order-theoretic viewpoint, the column and row signs determine a linear orders on the entries of $\pi^\gridded$ in each column and row.

For each index $k$, we define a linear \emph{column order} $\colorder_k$ on all of the entries of~$\pi$ that lie in the~$k$\th column of~$\pi^\gridded$. The orientation of $\colorder_k$ is dictated by $c_k$, in the sense that we order these entries from left to right if $c_k=1$ and right to left if $c_k=-1$. Thus if both $\pi(i)$ and $\pi(j)$ lie in column~$k$ of~$\pi^\gridded$, then
\[
	i\colorder_k j
	\iff
	\begin{cases}
		\text{$i\le j$, if $c_k=1$, or}\\
		\text{$i\ge j$, if $c_k=-1$.}
	\end{cases}
\]
Analogously, for each index $\ell$, we define a linear \emph{row order} $\roworder_\ell$ on all of the entries of~$\pi$ that lie in the $\ell$\th row of~$\pi^\gridded$. This order is oriented from bottom to top if $r_\ell=1$ and from top to bottom if~${r_\ell=-1}$, and thus if both~$\pi(i)$ and~$\pi(j)$ lie in row~$\ell$ of~$\pi^\gridded$, then
\[
	i\roworder_\ell j
	\iff
	\begin{cases}
		\text{$\pi(i)\le \pi(j)$, if $r_\ell=1$, or}\\
		\text{$\pi(i)\ge \pi(j)$, if $r_\ell=-1$.}
	\end{cases}
\]
We refer to the column and row orders together as the \emph{local orders} of~$\pi^\gridded$. 

\begin{figure}
\begin{footnotesize}
\begin{center}
	\begin{tikzpicture}[scale=1.4, baseline=(current bounding box.center)]
		\draw[step=1, darkgray, thick, line cap=round] (0,0) grid (3,2);
		\draw [->, very thick, line cap=round] (0,2)--(0.95,1.05);
		\draw [->, very thick, line cap=round] (2,0)--(1.05,0.95);
		\draw [->, very thick, line cap=round] (2,2)--(1.05,1.05);
		\draw [->, very thick, line cap=round] (3,2)--(2.05,1.05);
		\draw [->, very thick, line cap=round] (3,0)--(2.05,0.95);
		\plotpartialperm{1.8/1.8, 0.3/1.7, 2.6/0.4, 2.5/1.5, 2.4/0.6, 1.2/0.8, 0.7/1.3};
		\node at (1.8,1.8) [below right] {$7$};
		\node at (0.3,1.7) [above right] {$6$};
		\node at (2.6,0.4) [below left] {$1$};
		\node at (2.5,1.5) [above left] {$5$};
		\node at (2.4,0.6) [below left] {$2$};
		\node at (1.2,0.8) [below left] {$3$};
		\node at (0.7,1.3) [above right] {$4$};
		\draw[->] (-0.1,1.9)--(-0.1,1.1);
		\draw[->] (-0.1,0.1)--(-0.1,0.9);
		\draw[->] (0.1,-0.1)--(0.9,-0.1);
		\draw[->] (1.9,-0.1)--(1.1,-0.1);
		\draw[->] (2.9,-0.1)--(2.1,-0.1);
	\end{tikzpicture}
\quad\quad\quad
	\begin{tikzpicture}[scale=1.4, baseline=(current bounding box.center)]
		\draw[step=1, darkgray, thick, line cap=round] (0,0) grid (3,2);
		\draw [->, very thick, line cap=round] (0,2)--(0.95,1.05);
		\draw [->, very thick, line cap=round] (2,0)--(1.05,0.95);
		\draw [->, very thick, line cap=round] (2,2)--(1.05,1.05);
		\draw [->, very thick, line cap=round] (3,2)--(2.05,1.05);
		\draw [->, very thick, line cap=round] (3,0)--(2.05,0.95);
		\plotpartialperm{1.8/1.8, 0.3/1.7, 2.6/0.4, 2.5/1.5, 2.4/0.6, 1.2/0.8, 0.7/1.3};
		\node at (1.8,1.8) [below right] {$4$};
		\node at (0.3,1.7) [above right] {$1$};
		\node at (2.6,0.4) [below left] {$7$};
		\node at (2.5,1.5) [above left] {$6$};
		\node at (2.4,0.6) [below left] {$5$};
		\node at (1.2,0.8) [below left] {$3$};
		\node at (0.7,1.3) [above right] {$2$};
		\draw[->] (-0.1,1.9)--(-0.1,1.1);
		\draw[->] (-0.1,0.1)--(-0.1,0.9);
		\draw[->] (0.1,-0.1)--(0.9,-0.1);
		\draw[->] (1.9,-0.1)--(1.1,-0.1);
		\draw[->] (2.9,-0.1)--(2.1,-0.1);
	\end{tikzpicture}
\end{center}
\end{footnotesize}
\caption{Two drawings of the geometric gridding of~$\pi(1)\pi(2)\cdots\pi(7)=6437251$ from Figure~\ref{fig-ggc-example}, with arrows to indicate column and row signs. In the drawing on the left, the points are labelled by their vertical positions (that is, the values $\pi(i)$ of the corresponding entries of the permutation), while in the drawing on the right, they are labelled by their horizontal positions (that is, the indices of the corresponding entries of the permutation).}
\label{fig-ggc-example-with-signs}
\end{figure}
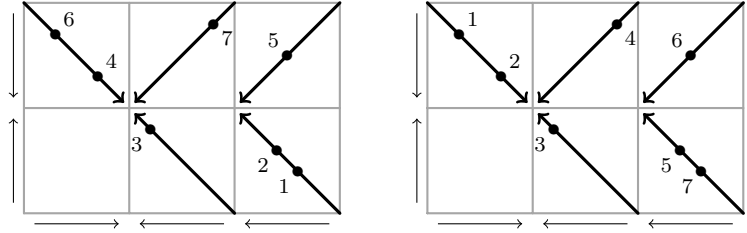

We pause to briefly consider these local orders for the gridding of our example in Figure~\ref{fig-ggc-example-with-signs}, with column and row signs indicated by the picture. Note that we typically discuss entries of permutations by \emph{value}, as on the left of Figure~\ref{fig-ggc-example-with-signs}, but we have defined the local orders by \emph{index}. This is why we have also supplied the drawing on the right of Figure~\ref{fig-ggc-example-with-signs}, which is identical to the one on the left, except that points are labelled by the index of their corresponding entry in the permutation.

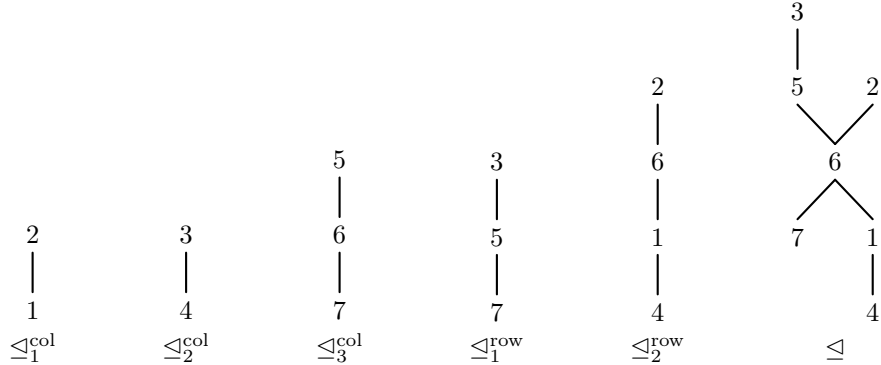
\begin{figure}
\begin{center}
	\begin{tikzpicture}[scale=0.5, every node/.style={}, baseline=(current bounding box.south)]
		\node (1) at (1,0) {$1$};
		\node (2) at (1,2) {$2$};
		\draw [thick, line cap=round] (1.north) -- (2.south);
		\node at (1,-1) {$\colorder_1$};
	\end{tikzpicture}
\quad\quad\quad
	\begin{tikzpicture}[scale=0.5, every node/.style={}, baseline=(current bounding box.south)]
		\node (4) at (1,0) {$4$};
		\node (3) at (1,2) {$3$};
		\draw [thick, line cap=round] (4.north) -- (3.south);
		\node at (1,-1) {$\colorder_2$};
	\end{tikzpicture}
\quad\quad\quad
	\begin{tikzpicture}[scale=0.5, every node/.style={}, baseline=(current bounding box.south)]
		\node (7) at (1,0) {$7$};
		\node (6) at (1,2) {$6$};
		\node (5) at (1,4) {$5$};
		\draw [thick, line cap=round] (7.north) -- (6.south);
		\draw [thick, line cap=round] (6.north) -- (5.south);
		\node at (1,-1) {$\colorder_3$};
	\end{tikzpicture}
\quad\quad\quad
	\begin{tikzpicture}[scale=0.5, every node/.style={}, baseline=(current bounding box.south)]
		\node (7) at (1,0) {$7$};
		\node (5) at (1,2) {$5$};
		\node (3) at (1,4) {$3$};
		\draw [thick, line cap=round] (7.north) -- (5.south);
		\draw [thick, line cap=round] (5.north) -- (3.south);
		\node at (1,-1) {$\roworder_1$};
	\end{tikzpicture}
\quad\quad\quad
	\begin{tikzpicture}[scale=0.5, every node/.style={}, baseline=(current bounding box.south)]
		\node (4) at (1,0) {$4$};
		\node (1) at (1,2) {$1$};
		\node (6) at (1,4) {$6$};
		\node (2) at (1,6) {$2$};
		\draw [thick, line cap=round] (4.north) -- (1.south);
		\draw [thick, line cap=round] (1.north) -- (6.south);
		\draw [thick, line cap=round] (6.north) -- (2.south);
		\node at (1,-1) {$\roworder_2$};
	\end{tikzpicture}
\quad\quad\quad
	\begin{tikzpicture}[scale=0.5, every node/.style={}, baseline=(current bounding box.south)]
		\node (7) at (1,2) {$7$};
		\node (5) at (1,6) {$5$};
		\node (3) at (1,8) {$3$};
		\node (4) at (3,0) {$4$};
		\node (1) at (3,2) {$1$};
		\node (6) at (2,4) {$6$};
		\node (2) at (3,6) {$2$};
		\draw [thick, line cap=round] (4.north) -- (1.south);
		\draw [thick, line cap=round] (1.north) -- (6.south);
		\draw [thick, line cap=round] (6.north) -- (2.south);
		\draw [thick, line cap=round] (7.north) -- (6.south);
		\draw [thick, line cap=round] (6.north) -- (5.south);
		\draw [thick, line cap=round] (5.north) -- (3.south);
		\node at (2,-1) {\phantom{$\roworder_2$}};
		\node at (2,-1) {$\trianglelefteq$};
	\end{tikzpicture}
\end{center}
\caption{Hasse diagrams showing the local orders for the gridded permutation shown in Figure~\ref{fig-ggc-example-with-signs}, together with the Hasse diagram of the transitive closure of their union on the far right.}
\label{fig-perm-hasse}
\end{figure}

As~$M$ has $3$ columns and $2$ rows, there are $5$ local orders, which we can read off as
\[
\begin{array}{lclclcl}
	1&\colorder_1&2,\\[2pt]
	4&\colorder_2&3,\\[2pt]
	7&\colorder_3&6&\colorder_3&5,\\[2pt]
	7&\roworder_1&5&\roworder_1&3,\\[2pt]
	4&\roworder_2&1&\roworder_2&6&\roworder_2&2.
\end{array}
\]

It is not an accident that the transitive closure of the union of the local orders happens to form a poset. For the gridded permutation from Figure~\ref{fig-ggc-example-with-signs}, the Hasse diagram of this poset is shown in Figure~\ref{fig-perm-hasse}, along with the Hasse diagrams of the local orders. To see that this is the case in general, we note that each local order actually orders the elements of~$\pi$ in increasing order of their distance to the base point of their respective cell; this is a feature of geometric griddings by partial multiplication matrices. 

To make this precise, let us say that a \emph{geometric realization} of the gridded permutation $\pi^\gridded$ is a drawing of the permutation~$\pi$ on the standard figure of~$M$ in such a way that every entry of~$\pi$ lies in the cell specified by~$\pi^\gridded$. Suppose that~$\pi^\gridded$ has a geometric realization, and fix one. For each index~$i$, let $d(i)$ denote the (euclidean) distance between the point of this geometric realization that corresponds to $\pi(i)$ and the base point of its cell (thus $d(i)$ lies between $0$ and $\sqrt{2}$ for every index~$i$). It is possible that points in independent cells (that is, cells that do not share a row or a column) have the same distance to their base points. If necessary, one can modify the realization to avoid this situation by perturbing the points in question---since the plot of the permutation is a generic set of points, one can easily check that this is always possible. We then have, if the entries~$\pi(i)$ and~$\pi(j)$ both lie in column~$k$ of~$\pi^\gridded$, that
\[
i\colorder_k j
\iff
d(i)\le d(j),
\]
while, analogously, if $\pi(i)$ and $\pi(j)$ both lie in row~$\ell$ of~$\pi^\gridded$, then
\[
i\roworder_\ell j
\iff
d(i)\le d(j).
\]

Recall that a set of linear orders is \emph{consistent} if the transitive closure of their union forms a poset, which happens if and only if their union does not contain a cycle. By our above observation, if $\pi^\gridded$ has a geometric realization, then the local orders of this gridded permutation must be consistent, because each of them agrees with the distance-to-base-point ordering, which is a linear ordering.

Conversely suppose that the local orders of the gridding $\pi^\gridded$ are consistent, and thus their union forms a poset, say $\trianglelefteq$. Fix a linear extension of this poset, which amounts to choosing a permutation $\psi$ of the indices of~$\pi$ such that
\[
	i\trianglelefteq j\implies \psi(i)\le \psi(j).
\]
Also choose an arbitrary set of distances
\[
	0<d_1<d_2<\cdots<d_n<\sqrt{2}.
\]
We may then obtain a geometric realization of this gridding of~$\pi$ by placing, for each index~$i$, a point corresponding to $\pi(i)$ at distance $d_{\psi(i)}$ from the base point in the appropriate cell of the standard figure of~$M$ (the cell in which $\pi(i)$ lies in~$\pi^\gridded$).

Returning to our example from Figure~\ref{fig-ggc-example-with-signs}, we have already drawn the Hasse diagram of the transitive closure of the union of the local orders. To obtain a linear extension of this poset, we may choose to place the points in the distance-to-base-point order $4$, $1$, $7$, $6$, $5$, $2$, $3$. This corresponds to choosing
\[
	\psi
	=
	\begin{pmatrix}
		1&2&3&4&5&6&7\\
		2&6&7&1&5&4&3
	\end{pmatrix},
\]
because then
\[
	\psi^{-1}
	=
	\begin{pmatrix}
		1&2&3&4&5&6&7\\
		4&1&7&6&5&2&3
	\end{pmatrix}.
\]
(In general, $\psi(i)$ tells us where $i$ lies in the linear extension, so reading $\psi^{-1}$ from left to right gives us the linear extension itself.) We have not specified how to choose the distances~$d_i$, but it is not hard to see that this choice is immaterial to the gridded permutation we are trying to produce.

In general, we have established the following result.

\begin{proposition}
\label{prop-geom-realization-consistent}
Let~$M$ be a partial multiplication matrix with fixed column and row signs, and let~$\pi^\gridded$ be an~$M$-gridding of a permutation $\pi\in\Grid(M)$. Then $\pi^\gridded$ has a geometric realization if and only if its local orders are consistent.	
\end{proposition}

This approach underlies the encoding of geometric grid classes by words, first presented in \cite{albert:geometric-grid-:}, and described briefly here. Define the \emph{cell alphabet} of the matrix~$M$ to be the set
\[
	\Sigma=\{a_{k\ell}\st M(k,\ell)\neq 0\},
\]
so that the letters $a_{k\ell}$ correspond to the nonzero entries of~$M$. We now describe a mapping $\varphi^\gridded$ from the language of words $\Sigma^\ast$ to the set of geometrically gridded members of $\Geom(M)$. For each word~$w=w(1)w(2)\cdots w(n)\in\Sigma^\ast$, we construct a permutation $\pi\in\Geom(M)$ as follows: choose a set $0<d_1<d_2<\cdots<d_n<\sqrt{2}$ of distances, and then for $1\le i\le n$, if $w(i)=a_{k\ell}$, place a point on the diagonal of the cell $(k,\ell)$ of the standard figure of~$M$ at (euclidean) distance $d_i$ from the base point of that cell.

As in our discussion above, the choice of distances is also immaterial here, and so the mapping $\varphi^\gridded$ is well-defined. Moreover, our previous discussion shows that $\varphi^\gridded$ is onto, by the following argument. Given any gridding $\pi^\gridded$ that has a geometric realization, Proposition~\ref{prop-geom-realization-consistent} shows that the local orders must be consistent. This means that the transitive closure of their union will have a linear extension, which we can take to define the permutation $\psi$ as before. We may then define the word~$w$ by the rule that
\[
	w(\psi(i))=a_{k\ell}
	\iff
	\text{$\pi(i)$ lies in cell $(k,\ell)$ in the gridding $\pi^\gridded$},
\]
and we see that $\varphi^\gridded(w)=\pi^\gridded$.

By simply removing the gridding from the gridded permutation $\varphi^\gridded(w)$, we obtain a permutation $\varphi(w)\in\Geom(M)$. It is easy to see that the mapping $\varphi$ is \emph{order-preserving}, in the sense that if $u$ is a not-necessarily-consecutive subword of $w$, then $\varphi(u)\le\varphi(w)$ as permutations. This observation and Higman's lemma~\cite{higman:ordering-by-div:} thereby imply that all geometrically griddable classes are well-quasi-ordered, a result of Albert, Atkinson, Bouvel, Ru\v{s}kuc, and Vatter~{\cite[Theorem~6.1]{albert:geometric-grid-:}}.

Our discussion shows that every different linear extension $\psi$ gives rise to a different word~$w$ that encodes the same gridded permutation $\pi^\gridded$. Indeed, part of the analysis of \cite{albert:geometric-grid-:} is the observation that changing the relative order in the linear extension $\psi$ of elements of independent cells is reflected, on the word level, by the set of encodings of a given gridded permutation $\pi^\gridded$ being described as a partially commutative monoid (also known as a trace monoid, such objects have been extensively studied, see Diekert's text \emph{Combinatorics on Traces}~\cite{diekert:combinatorics-o:} for example). That said, this further formalization is not necessary for our arguments.

Another detail that does not concern us in the present work is that members of $\Geom(M)$ may have many different~$M$-griddings with geometric realizations. For example, if we were to ask only that our words encode the \emph{permutation} shown in Figure~\ref{fig-ggc-example-with-signs}, then we would have many more, as it has several alternative griddings (for example, the point corresponding to $\pi(2)$ could shift to any of two other cells, or if that point stayed put, then the point corresponding to $\pi(3)$ could shift cells itself). Albert, Atkinson, Bouvel, Ru\v{s}kuc, and Vatter~{\cite[Theorem~8.1]{albert:geometric-grid-:}} show how to identify---in a regular manner---a distinguished geometric gridding for every member of $\Geom(M)$, thus establishing that every geometrically griddable class is in bijection with a regular language and thus has a rational generating function.

Our task here is different. Instead of trying to choose one geometric gridding for each member of a geometric grid class, we seek only to show that certain griddings \emph{are} geometric. Returning to this goal, the easier direction of Theorem~\ref{thm-ggc-lettericity}---which states that the graph class associated to a geometrically griddable permutation class has bounded lettericity---has already been proved in \cite{alecu:letter-graphs-a:abstract,alecu:letter-graphs-a:}, but we include another proof here, both for completeness and because it further illustrates some ideas used in the proof of the converse. Note that in the proof of this result, we are able to use the same word to encode both the permutation $\pi\in\Geom(M)$ and the letter graph satisfying~${\Gamma_D(w)\cong G_\pi}$.

\begin{theorem}[Alecu, Lozin, de Werra, and Zamaraev~\cite{alecu:letter-graphs-a:abstract,alecu:letter-graphs-a:}]
\label{thm-ggc-lettericity-half}
If the permutation class $\C$ is geometrically griddable, then the corresponding graph class $G_{\C}$ has bounded lettericity.
\end{theorem}
\begin{proof}
Suppose that $\C\subseteq\Geom(M)$ where~$M$ is a partial multiplication matrix with column signs~$(c_k)$ and row signs~$(r_\ell)$, and denote the cell alphabet of~$M$ by ${\Sigma=\{a_{k\ell}\st M(k,\ell)\neq 0\}}$. Let~${\pi\in\C}$ be arbitrary, and suppose that~$\pi=\varphi(w)$, where $w=w(1)\cdots w(n)\in\Sigma^\ast$. We prove the result by showing how to construct a decoder $D\subseteq\Sigma^2$ such that $\Gamma_D(w)\cong G_\pi$.

First, if we have~$M(k,\ell)=1$, then the entries in cell $(k,\ell)$ are increasing, and so we set $(a_{k\ell},a_{k\ell})\notin D$. Otherwise if~$M(k,\ell)=-1$, then the entries in cell $(k,\ell)$ are decreasing, and we let $(a_{k\ell},a_{k\ell})\in D$.

Next, if the cells $(k,\ell)$ and $(k',\ell')$ share neither a row nor a column (that is, if ${k\neq k'}$ and ${\ell\neq\ell'}$), then either every pair of entries with one from each of these two cells forms an inversion, in which case we set ${(a_{k\ell},a_{k'\ell'}),(a_{k'\ell'},a_{k\ell})\in D}$, or no such pair forms an inversion, in which case we set ${(a_{k\ell},a_{k'\ell'}),(a_{k'\ell'},a_{k\ell})\notin D}$.

Finally we must consider cells that share either a row or a column but not both. As the case where the cells share a row is symmetric, we consider entries lying in the cells $(k,\ell)$ and $(k,\ell')$, which share a column but lie in different rows. We may suppose without loss of generality that $\ell<\ell'$, so the entries of the cell $(k,\ell)$ lie below those of the cell $(k,\ell')$. Thus the only inversions between a pair of such entries consist of an entry in the cell $(k,\ell')$ lying to the left of an entry from the cell $(k,\ell)$. If $c_k=1$, then the entries of this column are encoded from left to right, so we set $(a_{k\ell'},a_{k\ell})\in D$ and $(a_{k\ell},a_{k\ell'})\notin D$. Otherwise, if $c_k=-1$, then these entries are encoded from right to left, so we set $(a_{k\ell},a_{k\ell'})\in D$ and $(a_{k\ell'},a_{k\ell})\notin D$.
\end{proof}

\section{Proof of the Main Result}
\label{sec-main-proof}

In this section we prove the other direction of Theorem~\ref{thm-ggc-lettericity}. The bulk of this proof is conducted on the level of an individual permutation, and at the heart of our approach to individual permutations is the following fact, formalized in Proposition~\ref{prop-one-perm-final-result}: if a permutation~$\pi$ has a monotone gridding by a matrix~$M$, then~$\pi$ is geometrically griddable by a matrix~$M^\ggridded$ whose size only depends on the size of~$M$ and the lettericity of its inversion graph $G_\pi$.

The proof of this fact is somewhat technical, but the concept behind it is quite transparent. A large amount of structural information about the permutation~$\pi$ is available to us. This information comes from two sources: the monotone gridding by~$M$ on the one hand, and the letter graph representation of $G_\pi$ on the other. The problem is that it is not a priori clear how to interpret this information simultaneously. There is, however, a straightforward solution: we take a ``common refinement'' of the available information, and use the word that encodes $G_\pi$ as a letter graph to define a linear extension $\psi$ on the entries of~$\pi$ as in the previous section. In practice, the way this is achieved is by first refining the alphabet and the decoder of the letter graph to ensure that each letter is used in only one cell. We then refine the gridding by adding vertical and horizontal lines surrounding the points encoded by each letter. The bulk of the work then consists of showing that this refinement behaves as we would like it to.

We begin our proof by fixing a bijection between the inversion graph of~$\pi$ and an isomorphic letter graph, and using this to define a linear extension $\psi$ on the entries of~$\pi$. Our initial results, Propositions~\ref{prop-distinguish-perms} and \ref{prop-separate-cells} in particular, establish that $\psi$ behaves as we would like it to ``locally''. Later results then broaden the scope until Proposition~\ref{prop-one-perm-final-result} concludes our analysis of a single permutation, at which point the proof of our main result is quickly completed in Section~\ref{subsec-end-of-proof}.

Suppose that~$\pi$ is a permutation of length $n$ and that
\[
	G_\pi\cong \Gamma_D(w)
\]
for some finite alphabet $\Sigma$, some decoder $D\subseteq\Sigma^2$, and some word~$w\in\Sigma^n$. We consider a particular isomorphism between $G_\pi$ and $\Gamma_D(w)$, which we denote by
\[
	\psi\st V(G_\pi) \to V(\Gamma_D(w)).
\]
One may notice that we have given this isomorphism the same name we gave the linear extension in Section~\ref{sec-ggc-encoding}. This is intentional, as the ultimate result of this section states that they play the same role%
\footnote{This is also the same role as played by the index correspondence used in Albert, Ru\v{s}kuc, and Vatter~\cite[Section~3]{albert:inflations-of-g:}.}.
We say that the entry~$\pi(i)$ \emph{corresponds} to the vertex $\psi(i)$ of $\Gamma_D(w)$ and is \emph{encoded} by the letter $w(\psi(i))$, which we denote symbolically by writing
\[
	\pi(i)
	\quad\longleftrightarrow\quad
	\text{vertex $\psi(i)$ in $\Gamma_D(w)$}
	\quad\longleftrightarrow\quad
	w(\psi(i)).
\]

%
%

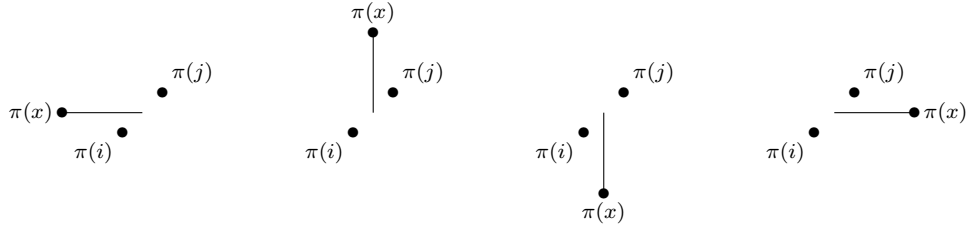
\begin{figure}
\begin{center}
	\begin{tikzpicture}[scale={2.8/10.5}, baseline=(current bounding box.center)]
		\node at (-1,-1) [below left] {\footnotesize $\pi(i)$};
		\node at (1,1) [above right] {\footnotesize $\pi(j)$};
		\plotpartialperm{-1/-1, 1/1, -4/0};
		\draw (-4,0)--(0,0);
		\node at (-4,0) [left] {\footnotesize $\pi(x)$};
		\node [color=white] at (0,4) [above] {\footnotesize $\pi(x)$};
		\node [color=white] at (0,-4) [below] {\footnotesize $\pi(x)$};
	\end{tikzpicture}
	\begin{tikzpicture}[scale={2.8/10.5}, baseline=(current bounding box.center)]
		\node at (-1,-1) [below left] {\footnotesize $\pi(i)$};
		\node at (1,1) [above right] {\footnotesize $\pi(j)$};
		\plotpartialperm{-1/-1, 1/1, 0/4};
		\draw (0,4)--(0,0);
		\node [color=white] at (-4,0) [left] {\footnotesize $\pi(x)$};
		\node at (0,4) [above] {\footnotesize $\pi(x)$};
		\node [color=white] at (0,-4) [below] {\footnotesize $\pi(x)$};
	\end{tikzpicture}
	\begin{tikzpicture}[scale={2.8/10.5}, baseline=(current bounding box.center)]
		\node at (-1,-1) [below left] {\footnotesize $\pi(i)$};
		\node at (1,1) [above right] {\footnotesize $\pi(j)$};
		\plotpartialperm{-1/-1, 1/1, 0/-4};
		\draw (0,-4)--(0,0);
		\node [color=white] at (-4,0) [left] {\footnotesize $\pi(x)$};
		\node [color=white] at (0,4) [above] {\footnotesize $\pi(x)$};
		\node at (0,-4) [below] {\footnotesize $\pi(x)$};
	\end{tikzpicture}
	\begin{tikzpicture}[scale={2.8/10.5}, baseline=(current bounding box.center)]
		\node at (-1,-1) [below left] {\footnotesize $\pi(i)$};
		\node at (1,1) [above right] {\footnotesize $\pi(j)$};
		\plotpartialperm{-1/-1, 1/1, 4/0};
		\draw (4,0)--(0,0);
		\node [color=white] at (-4,0) [left] {\footnotesize $\pi(x)$};
		\node [color=white] at (0,4) [above] {\footnotesize $\pi(x)$};
		\node [color=white] at (0,-4) [below] {\footnotesize $\pi(x)$};
		\node at (4,0) [right] {\footnotesize $\pi(x)$};
	\end{tikzpicture}
\quad\quad\quad
\end{center}
\caption{In all four cases, the entry~$\pi(x)$ separates $\pi(i)$ from $\pi(j)$.}
\label{fig-separates}
\end{figure}

As shown in Figure~\ref{fig-separates}, we say that the entries $\pi(i)$ and $\pi(j)$ of the permutation~$\pi$ are \emph{separated} if~$\pi$ contains a third entry lying between them either horizontally or vertically, but not both, and we say that this third entry separates $\pi(i)$ \emph{from} $\pi(j)$. The analogous concept in the graph context is that two vertices are \emph{distinguished} if there is a third vertex that is adjacent to one but not both of them. It is easy to see that the entries $\pi(i)$ and $\pi(j)$ are separated in~$\pi$ if and only if the vertices $i$ and $j$ are distinguished in $G_\pi$.

The connection between separation and lettering is given below.

\begin{proposition}
\label{prop-distinguish-perms}
Suppose that $\psi\st V(G_\pi) \to V(\Gamma_D(w))$ is an isomorphism. If~$\pi(i_1)$ and~$\pi(i_2)$ are encoded by the same letter and are separated by~$\pi(x)$, then ${\psi(i_1)<\psi(x)<\psi(i_2)}$ or the reverse.
\end{proposition}
\begin{proof}
Since $\pi(x)$ separates $\pi(i_1)$ and $\pi(i_2)$, the vertex $\psi(x)$ distinguishes the vertices $\psi(i_1)$ and $\psi(i_2)$. We assume that $\psi(i_1)<\psi(i_2)$; if instead the reverse is true then the result follows by a symmetrical argument. If $\psi(x)<\psi(i_1)<\psi(i_2)$ or $\psi(i_1)<\psi(i_2)<\psi(x)$, then the vertex $\psi(x)$ would have to be adjacent to either both vertices $\psi(i_1)$ and $\psi(i_2)$ or neither, because they are encoded by the same letter.
\end{proof}

\subsection{An initial gridding}

We now suppose that~$\pi\in\Grid(M)$ for some~$\zpm$~matrix~$M$ of size~$t\times u$, and we denote by~$\pi^\gridded$ the permutation~$\pi$ together with a fixed~$M$-gridding.
Note that~$\pi^\gridded$ is only assumed to be a monotone gridding (if it is in fact geometric, then we are already done).

\begin{figure}
\begin{center}
	\begin{tikzpicture}[scale={2.8/10.5}, baseline=(current bounding box.center)]
		\plotpermbox{0.5}{0.5}{11}{4};
		\plotpermbox{0.5}{6.5}{11}{9.5};
		\plotpartialperm{1/1,3/9,5/3,7/7};
		\draw (3,9)--(3,2);
		\draw (7,7)--(7,4);
		\node at (1,1) [right] {\footnotesize $\pi(i_1)$};
		\node at (3,9) [right] {\footnotesize $\pi(i_2)$};
		\node at (5,3) [right] {\footnotesize $\pi(i_3)$};
		\node at (7,7) [right] {\footnotesize $\pi(i_4)$};
	\end{tikzpicture}
\end{center}
\caption{The situation to which Proposition~\ref{prop-separate-cells} applies.}
\label{fig-prop-separate-cells}
\end{figure}
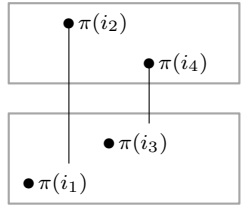

The following result, which follows routinely from Proposition~\ref{prop-distinguish-perms}, applies to the arrangement shown in Figure~\ref{fig-prop-separate-cells}. This result has a symmetric version that applies when the four entries alternate when read by value (in the version stated below, the entries alternate when read by position).

\begin{proposition}
\label{prop-separate-cells}
Suppose that $\psi\st V(G_\pi) \to V(\Gamma_D(w))$ is an isomorphism, that~$\pi^\gridded$ is a monotonically gridded permutation, and that we have
\begin{itemize}[topsep=-4pt, itemsep=-2pt]
	\item $i_1<i_2<i_3<i_4$,
	\item $\pi(i_1)$ and $\pi(i_3)$ are encoded by the same letter and lie together in one cell of~$\pi^\gridded$, and
	\item $\pi(i_2)$ and $\pi(i_4)$ are encoded by the same letter and lie together in a different cell of~$\pi^\gridded$.
\end{itemize}
Then $\psi(i_1)<\psi(i_2)<\psi(i_3)<\psi(i_4)$ or the reverse.
\end{proposition}
\begin{proof}
Because $i_1<i_2<i_3$ and $\pi(i_2)$ lies in a different cell from $\pi(i_1)$ and $\pi(i_3)$, it separates these entries. Proposition~\ref{prop-distinguish-perms} then implies that either $\psi(i_1)<\psi(i_2)<\psi(i_3)$ or the reverse; as in our previous proof, if the reverse holds, then the result follows by a symmetrical argument.

Similarly, because $i_2<i_3<i_4$ and $\pi(i_3)$ lies in a different cell from $\pi(i_2)$ and $\pi(i_4)$, $\pi(i_3)$ separates these entries. Thus Proposition~\ref{prop-distinguish-perms} implies that $\psi(i_2)<\psi(i_3)<\psi(i_4)$, because the reverse would contradict the fact that $\psi(i_2)<\psi(i_3)$. It follows that $\psi(i_1)<\psi(i_2)<\psi(i_3)<\psi(i_4)$, as claimed.
\end{proof}


%
%
%
\subsection{Monotone intervals}
\label{subsec-mono-intervals}

There may be many choices for $\psi$. Indeed, if we consider the identity $\pi=12\cdots n$, the constant word~$w=\mathsf{aa}\cdots\mathsf{a}$, and an empty decoder $D=\emptyset$, then \emph{every} bijection ${\psi\st[n]\to[n]}$ is an isomorphism between $G_{12\cdots n}\cong\overline{K}_n$ and $\Gamma_D(w)\cong\overline{K}_n$. Thus we need to impose some conditions on~$\pi$ in order to get better control of the bijection $\psi$.

An \emph{increasing interval} in a permutation is a sequence of consecutive entries such that each of the entries is $1$ more than the previous entry, while a \emph{decreasing interval} is, analogously, a sequence of consecutive entries such that each is $1$ less than the previous entry. A \emph{monotone interval} is then either an increasing or a decreasing interval. The empty set and single entries by themselves are considered to be \emph{trivial} monotone intervals, while larger monotone intervals are \emph{nontrivial}. If the permutation~$\pi$ has a nontrivial monotone interval, then we can \emph{contract} this monotone interval into a single entry to obtain a permutation $\sigma$. Conversely, we can \emph{inflate} $\sigma$ (replacing the appropriate entry by a monotone interval) to obtain~$\pi$.

Suppose that~$\pi\in\Grid(M)$ were to have one or more nontrivial monotone intervals, and fix a particular~$M$-gridding $\pi^\gridded$ of it. We may contract its monotone intervals to obtain a permutation~${\sigma\in\Grid(M)}$ with corresponding~$M$-gridding $\sigma^\gridded$ (arising from $\pi^\gridded$). If $\sigma^\gridded$ has a geometric realization, then it follows that $\pi^\gridded$ has a geometric realization, so $\pi\in\Geom(M)$ whenever $\sigma\in\Geom(M)$. We record this fact below for later reference.

\begin{observation}
\label{obs-mono-interval-gridding}
Suppose that the~$M$-gridded permutation $\sigma^\gridded$ can be obtained by contracting monotone intervals of the~$M$-gridded permutation $\pi^\gridded$. If $\sigma^\gridded$ has a geometric realization, then~$\pi^\gridded$ also has a geometric realization.
\end{observation}

Because of this observation, we may (and indeed, we need to) insist in what follows that~$\pi$ \emph{has only trivial monotone intervals}%
\footnote{We choose not to give such permutations a name, instead repeating this slightly cumbersome phrase, although these permutations have been given names in the literature. Atkinson and Stitt~\cite{atkinson:restricted-perm:wreath} refer to them as \emph{strongly irreducible}, while Brignall, Jel\'{\i}nek, Kyn\v{c}l, and Marchant~\cite{brignall:zeros-of-the-mo:} call them \emph{adjacency-free}.}.
We now begin to relate the position of entries of~$\pi$ to the order of their encodings in $w$, as in the result below.

\begin{proposition}
\label{prop-4-dist}
Suppose that $\psi\st V(G_\pi) \to V(\Gamma_D(w))$ is an isomorphism, that~$\pi^\gridded$ is a monotonically gridded permutation with only trivial monotone intervals, and that there are indices ${i_1<i_2<i_3}$ for which the entries $\pi(i_1)$, $\pi(i_2)$, and $\pi(i_3)$ all lie in the same cell of~$\pi^\gridded$ and are all encoded by the same letter. Then $\psi(i_1)<\psi(i_2)<\psi(i_3)$ or the reverse.
\end{proposition}
\begin{figure}
\begin{center}
	\begin{tikzpicture}[scale={2.8/8.5}, baseline=(current bounding box.center)]
		\plotpermbox{0.5}{0.5}{11.5}{7.5};
		\plotpermbox{0.9}{0.9}{7.6}{5.1};
		\plotpartialperm{1/1,3/3,5/5,2/7,9/4};
		\draw (2,7)--(2,2);
		\draw (9,4)--(4,4);
		\node at (1,1) [right] {\footnotesize $\pi(i_1)$};
		\node at (3,3) [right] {\footnotesize $\pi(i_2)$};
		\node at (5,5) [right] {\footnotesize $\pi(i_3)$};
		\node at (2,7) [right] {\footnotesize $\pi(x)$};
		\node at (9,4) [right] {\footnotesize $\pi(y)$};
	\end{tikzpicture}
\end{center}
\caption{An example of the situation in Proposition~\ref{prop-4-dist}.}
\label{fig-prop-4-dist}
\end{figure}
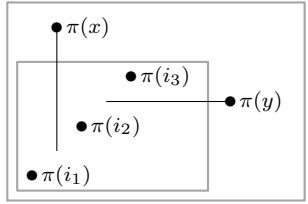

\begin{proof}
Because $\pi(i_1)$, $\pi(i_2)$, and $\pi(i_3)$ all lie in the same cell of~$\pi^\gridded$, they form a monotone subsequence. Because~$\pi$ has only trivial monotone intervals, there must be an entry~$\pi(x)$ separating $\pi(i_1)$ and $\pi(i_2)$ and an entry~$\pi(y)$ separating $\pi(i_2)$ and $\pi(i_3)$. Thus our situation resembles that shown in Figure~\ref{fig-prop-4-dist}.

Proposition~\ref{prop-distinguish-perms} implies that either $\psi(i_1)<\psi(x)<\psi(i_2)$ or the reverse. We assume this this inequality holds and prove that ${\psi(i_1)<\psi(i_2)<\psi(i_3)}$; if the reverse holds instead, then a symmetrical argument establishes that ${\psi(i_1)>\psi(i_2)>\psi(i_3)}$.

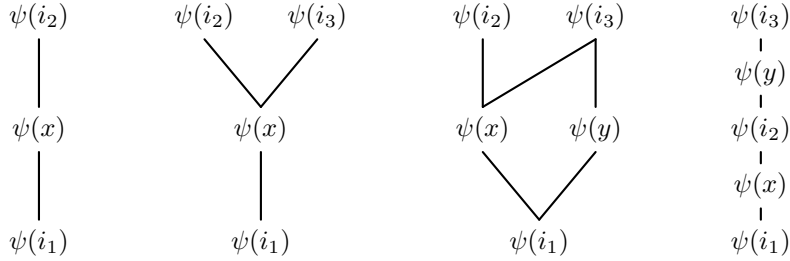
\begin{figure}
\begin{center}
	\begin{tikzpicture}[scale=0.75, every node/.style={}, baseline=(current bounding box.center)]
		\node (i) at (1,0) {$\psi(i_1)$};
		\node (x) at (1,2) {$\psi(x)$};
		\node (j) at (1,4) {$\psi(i_2)$};
		\draw [thick, line cap=round] (i.north) -- (x.south);
		\draw [thick, line cap=round] (x.north) -- (j.south);
	\end{tikzpicture}
\quad\quad\quad
	\begin{tikzpicture}[scale=0.75, every node/.style={}, baseline=(current bounding box.center)]
		\node (i) at (1,0) {$\psi(i_1)$};
		\node (x) at (1,2) {$\psi(x)$};
		\node (j) at (0,4) {$\psi(i_2)$};
		\node (k) at (2,4) {$\psi(i_3)$};
		\draw [thick, line cap=round] (i.north) -- (x.south);
		\draw [thick, line cap=round] (x.north) -- (j.south);
		\draw [thick, line cap=round] (x.north) -- (k.south);
	\end{tikzpicture}
\quad\quad\quad
	\begin{tikzpicture}[scale=0.75, every node/.style={}, baseline=(current bounding box.center)]
		\node (i) at (1,0) {$\psi(i_1)$};
		\node (x) at (0,2) {$\psi(x)$};
		\node (y) at (2,2) {$\psi(y)$};
		\node (j) at (0,4) {$\psi(i_2)$};
		\node (k) at (2,4) {$\psi(i_3)$};
		\draw [thick, line cap=round] (i.north) -- (x.south);
		\draw [thick, line cap=round] (i.north) -- (y.south);
		\draw [thick, line cap=round] (x.north) -- (j.south);
		\draw [thick, line cap=round] (x.north) -- (k.south);
		\draw [thick, line cap=round] (y.north) -- (k.south);
	\end{tikzpicture}
\quad\quad\quad
	\begin{tikzpicture}[scale=0.75, every node/.style={}, baseline=(current bounding box.center)]
		\node (i) at (1,0) {$\psi(i_1)$};
		\node (x) at (1,1) {$\psi(x)$};
		\node (j) at (1,2) {$\psi(i_2)$};
		\node (y) at (1,3) {$\psi(y)$};
		\node (k) at (1,4) {$\psi(i_3)$};
		\draw [thick, line cap=round] (i.north) -- (x.south);
		\draw [thick, line cap=round] (x.north) -- (j.south);
		\draw [thick, line cap=round] (j.north) -- (y.south);
		\draw [thick, line cap=round] (y.north) -- (k.south);
	\end{tikzpicture}
\end{center}
\caption{The evolving Hasse diagrams showing what we know about the ordering of the values of $\psi$ in the proof of Proposition~\ref{prop-4-dist}.}
\label{fig-perm-hasse-2}
\end{figure}

There are three more implications of Proposition~\ref{prop-distinguish-perms}, and Figure~\ref{fig-perm-hasse-2} shows how each of these implications updates our knowledge of the Hasse diagram of the indices $\psi(i_1)$, $\psi(i_2)$, $\psi(i_3)$, $\psi(x)$, and $\psi(y)$. We must have $\psi(i_1)<\psi(x)<\psi(i_3)$ because the reverse would contradict the fact that~$\psi(i_1)<\psi(x)$, leading to the second diagram of Figure~\ref{fig-perm-hasse-2}. We must then have $\psi(i_1)<\psi(y)<\psi(i_3)$ because the reverse would contradict the fact that~${\psi(i_1)<\psi(i_3)}$, leading to the third diagram of Figure~\ref{fig-perm-hasse-2}. Finally, we must have $\psi(i_2)<\psi(y)<\psi(i_3)$ because the reverse would contradict the fact that~${\psi(y)<\psi(i_3)}$, leading to the fourth diagram of Figure~\ref{fig-perm-hasse-2}.

We started by assuming (by symmetry) that $\psi(i_1)<\psi(i_2)$, and our final string of inequalities shows that $\psi(i_2)<\psi(i_3)$, completing the proof of the result.
\end{proof}

\subsection{Relettering}
\label{subsec-relettering}

To briefly recap our situation, we have a monotonically gridded permutation $\pi^\gridded$ of length $n$ with only trivial monotone intervals, an alphabet $\Sigma$, a decoder $D\subseteq\Sigma^2$, and a word~$w\in\Sigma^n$ such that~${G_\pi\cong \Gamma_D(w)}$. Further, we are considering a particular isomorphism $\psi\st V(G_\pi)\to V(\Gamma_D(w))$. We say that this collection of objects (the monotonically gridded permutation, the word, and the isomorphism) has the \emph{one cell per letter property} if for every letter~${a\in\Sigma}$, all entries encoded by~$a$ lie in the same cell of~$\pi^\gridded$. Put another way, this property says that if $w(\psi(i_1))=w(\psi(i_2))$, then~$\pi^\gridded(i_1)$ and~$\pi^\gridded(i_2)$ lie in the same cell.

We can always obtain the one cell per letter property at the expense of a larger alphabet, a process that we call \emph{relettering}. The first step is to expand our alphabet $\Sigma$ to a set
\[
	\Sigma^\gridded
	=
	\{(a,k,\ell)\st\text{$a\in\Sigma$ and~$M(k,\ell)\neq 0$}\}
\]
that includes cell information. Note that if $\pi^\gridded$ is a $t\times u$ gridding, then
\(
	| \Sigma^\gridded |
	\le
	tu \, |\Sigma|.
\)
Second, we add this information to our word~$w$, producing a new word~$w^\gridded\in(\Sigma^\gridded)^\ast$ in which
\[
	w^\gridded(\psi(i))=(a,k,\ell)
	\iff
	\text{$w(\psi(i))=a$ and $\pi(i)$ lies in cell $(k,\ell)$ of~$\pi^\gridded$.}
\]
Finally, we update the decoder to essentially ignore the extra information in the new letters, defining
\[
	D^\gridded
	=
	\{((a,i,j), (b,k,\ell))\st\text{$(a,b)\in D$,~$M(i,j)\neq 0$, and~$M(k,\ell)\neq 0$}\}.
\]

%
%

If we had an isomorphism $\psi\st V(G_\pi) \to V(\Gamma_D(w))$ before the relettering, then after the relettering, by a very slight abuse of notation, we view this also as an isomorphism $\psi\st V(G_\pi)\to V(\Gamma_{D^\gridded}(w^\gridded))$. We may do so because the relettering does not affect the letter graph or its labels at all: we have
\[
	V(\Gamma_{D^\gridded}(w^\gridded))=V(\Gamma_D(w))=\{1,2,\dots,n\}
\]
and
\[
	ij\in E(\Gamma_{D^\gridded}(w^\gridded))
	\iff
	ij\in E(\Gamma_D(w)).
\]

%
%

As we are assuming that~$\pi$ has only trivial monotone intervals, once we have done this relettering (to achieve the one cell per letter property), Proposition~\ref{prop-4-dist} implies that entries of the same letter ``read in the same order'' in the following sense. For each letter~$a\in\Sigma^\gridded$, the entries encoded by~$a$ all lie in a single cell of~$\pi^\gridded$. If these entries have indices ${i_1<i_2<\cdots<i_s}$, then we have
\[
	\psi(i_1)<\psi(i_2)<\cdots<\psi(i_s)
	\quad\text{or}\quad
	\psi(i_s)<\psi(i_{s-1})<\cdots<\psi(i_1).
\]
This follows from Proposition~\ref{prop-4-dist} for~$s\ge 3$, and is trivial for~$s\le 2$.

If ${\psi(i_1)<\cdots<\psi(i_s)}$, then these entries are encoded from left to right in the word~$w$, while if~${\psi(i_s)<\cdots<\psi(i_1)}$, then they are encoded from right to left. We refer to this order as the \emph{horizontal reading order of the letter}. If there is only one entry encoded by some letter, then we may choose the horizontal reading order arbitrarily to be either left-to-right or right-to-left.

Moreover, since
the entries encoded by $a$ 
form a monotone subsequence of~$\pi$ (because they lie together in a single cell of the monotonically gridded permutation $\pi^\gridded$),
they are also encoded in $w$ either from bottom to top or from top to bottom. We refer to this order as the \emph{vertical reading order of the letter}. (Again, if there is only one such entry, then its vertical reading order may be chosen arbitrarily.) Note that the vertical (resp., horizontal) reading order of a letter is completely determined by its horizontal (resp., vertical) reading order together with the sign of the cell in which the entries encoded by that letter lie; for example, if the entries encoded by some letter lie in a decreasing cell and are encoded from bottom to top, then this means that they are also encoded from right to left.

We refer to the combination of the horizontal and vertical reading orders of a letter as simply the \emph{reading order of the letter}. By our comments above, if we have the one cell per letter property, then the reading order of a letter is completely determined by its horizontal or vertical reading order and the sign of the cell in which the entries encoded by that letter lie.

\subsection{Refining the gridding}

Our next step is to refine the monotone gridding $\pi^\gridded$, by slicing it with additional vertical and horizontal lines, in a manner similar to that employed in \cite[Section~5]{vatter:small-permutati:}. The immediate aim of this additional slicing is to obtain the property in Proposition~\ref{prop-splitting-entries}, which allows us to appeal to Proposition~\ref{prop-separate-cells} in proving the subsequent results of this section.

The \emph{rectangular hull} of a set of points in the plane is the smallest axis-parallel rectangle containing them. Suppose that we have our monotone gridded permutation $\pi^\gridded$ and that we have relettered in order to achieve the one cell per letter property. For each letter $(a,k,\ell)\in\Sigma^\gridded$, let $R_{(a,k,\ell)}$ denote the rectangular hull of the entries of~$\pi^\gridded$ encoded by that letter. We now refine the gridding by adding, for each rectangle $R_{(a,k,\ell)}$, four lines that lie just outside its four sides. Thus for each letter~${(a,k,\ell)\in\Sigma^\gridded}$, we add two new horizontal lines and two new vertical lines. Consequently, because~$\pi^\gridded$ is a~$t\times u$ gridding, the maximum possible size of the new gridding is ${(t+2|\Sigma^\gridded|)}\times {(u+2|\Sigma^\gridded|)}$, which is at most $t(1+2u|\Sigma|)\times u(1+2t|\Sigma|)$ by our previous bound on~$|\Sigma^\gridded|$.

As the original gridding is denoted by the sharp symbol~$\raisebox{1pt}{\scalebox{1.5}{\gridded}}$, we use the double sharp symbol~$\raisebox{-0.4pt}{\scalebox{1.5}{\ggridded}}$ to denote the \emph{regridded permutation}~$\pi^\ggridded$. As stated above, the regridding gives us the following property.

\begin{proposition}
\label{prop-splitting-entries}
Let $\pi^\gridded$ and $\pi^\ggridded$ be as above and suppose that for some indices $i_2<i_3$, the entries~$\pi(i_2)$ and~$\pi(i_3)$ lie in the same column of $\pi^\ggridded$ but are encoded by different letters of $\Sigma^\gridded$. Then there is an index~$i_1$ satisfying ${i_1<i_2<i_3}$ so that $\pi(i_1)$ and $\pi(i_3)$ are encoded by the same letter of~$\Sigma^\gridded$.
\end{proposition}
\begin{proof}
Let~$\pi(i_1)$ denote the leftmost entry of $\pi^\gridded$ that is encoded by the same letter of~$\Sigma^\gridded$ as~$\pi(i_3)$. In forming the regridded permutation $\pi^\ggridded$, we sliced by a line that lies just to the left of~$\pi(i_1)$. Since~$\pi(i_2)$ and~$\pi(i_3)$ lie in the same column of $\pi^\ggridded$, this line must lie to the left of $\pi(i_2)$, and thus we must have ${i_1<i_2<i_3}$.
\end{proof}

It is worth noting that the entries $\pi(i_1)$ and $\pi(i_3)$ of Proposition~\ref{prop-splitting-entries} must lie in the same cell of $\pi^\gridded$ since they are encoded by the same letter of~$\Sigma^\gridded$, but that they need not lie in the same cell of~$\pi^\ggridded$.

We now consider the reading orders of letters in the regridded permutation $\pi^\ggridded$. Note that~$\pi^\ggridded$ likely no longer has the one cell per letter property%
\footnote{We could of course enlarge of our alphabet again to regain the one cell per letter property, but there is no need.},
but that property was only necessary to obtain reading orders for the letters of~$\Sigma^\gridded$. In the regridded permutation~$\pi^\ggridded$, we retain the reading orders of the letters of~$\Sigma^\gridded$ from the relettering of Section~\ref{subsec-relettering}. Further note that if the reading order of a letter of~$\Sigma^\gridded$ was chosen arbitrarily before, then that was because the letter encoded a single entry, and so in~$\pi^\ggridded$, that entry is isolated in a cell by itself that shares neither a column nor a row with another nonempty cell. We now argue that these reading orders are compatible with each other.

\begin{proposition}
\label{prop-two-letters-one-direction-column}
Suppose that~$\pi$ has only trivial monotone intervals and that $w^\gridded$, $\psi$, $\pi^\gridded$ and $\pi^\ggridded$ are as above. Then the letters of~$\Sigma^\gridded$ encoding any pair of entries lying in the same column of~$\pi^\ggridded$ share the same horizontal reading order. Thus, by symmetry, the letters of~$\Sigma^\gridded$ encoding any pair of entries lying in the same row of~$\pi^\ggridded$ also share the same vertical reading order.
\end{proposition}
\begin{proof}
Suppose that for some indices~${i_2<i_3}$, the entries~$\pi(i_2)$ and~$\pi(i_3)$ lie in the same column of~$\pi^\ggridded$. We want to show that the letters encoding these two entries have the same horizontal reading order. If these entries are encoded by the same letter, then the result is trivial, so we assume that they are encoded by different letters.

First suppose that~$\pi(i_2)$ and~$\pi(i_3)$ lie in different cells of~$\pi^\gridded$. Proposition~\ref{prop-splitting-entries} shows directly that there is an index $i_1$ satisfying ${i_1<i_2<i_3}$ so that $\pi(i_1)$ and $\pi(i_3)$ are encoded by the same letter of~$\Sigma^\gridded$. By symmetry, it also implies that there is an index $i_4$ satisfying ${i_2<i_3<i_4}$ so that $\pi(i_2)$ and $\pi(i_4)$ are encoded by the same letter of~$\Sigma^\gridded$. As we are assuming that~$\pi(i_2)$ and~$\pi(i_3)$ do not share a cell in~$\pi^\gridded$, the hypotheses of Proposition~\ref{prop-separate-cells} are satisfied. Thus ${\psi(i_1)<\psi(i_2)<\psi(i_3)<\psi(i_4)}$ or the reverse, and this proves that the horizontal reading orders of the corresponding letters agree. By symmetry, we have the analogous statement about entries that lie in the same row of~$\pi^\ggridded$ but not the same cell of~$\pi^\gridded$.

To complete the proof, suppose that the entries $\pi(i_2)$ and $\pi(i_3)$ lie in the same cell of~$\pi^\gridded$. Because~$\pi$ has only trivial monotone intervals, these two entries must be separated, say by the entry~$\pi(x)$. Thus $\pi(x)$ must lie in a different cell of~$\pi^\gridded$ from $\pi(i_2)$ and $\pi(i_3)$, but in the same column or row of~$\pi^\gridded$ as these entries.

If $\pi(x)$ lies in the same column as~$\pi(i_2)$ and $\pi(i_3)$, then, by the previous case applied to the pair $\pi(x)$ and $\pi(i_2)$, we see that the horizontal reading order of the letter encoding $\pi(x)$ agrees with the horizontal reading order of the letter encoding $\pi(i_2)$. Similarly, the horizontal reading order of the letter encoding $\pi(x)$ agrees with the horizontal reading order of the letter encoding $\pi(i_3)$. Thus, all three horizontal reading orders agree. Because $\pi(i_2)$ and $\pi(i_3)$ lie in the same cell of~$\pi^\gridded$, this forces the vertical reading orders of the letters encoding these entries to agree as well. A symmetrical argument handles the case where $\pi(x)$ lies in the same row as~$\pi(i_2)$ and $\pi(i_3)$. 
\end{proof}

With our next result we show that the regridded permutation $\pi^\ggridded$ has a geometric realization, which implies that~$\pi$ itself lies in the geometric grid class of a matrix of bounded size.

\begin{proposition}
\label{prop-one-perm-final-result}
Suppose that~$\pi$ has only trivial monotone intervals, that ${\pi\in\Grid(M)}$ for a~$\zpm$ matrix~$M$ of size~$t\times u$, and that $G_\pi\cong \Gamma_D(w)$ where $|\Sigma|\le r$, $D\subseteq\Sigma^2$, and $w\in\Sigma^\ast$. Then there is a~$\zpm$~matrix~$M^\ggridded$ of size at most ${t(1+2ur)}\times {u(1+2tr)}$ for which~$\pi\in\Geom(M^\ggridded)$.
\end{proposition}
\begin{proof}
Let $\psi\st V(G_\pi)\to V(\Gamma_D(w))$ be an isomorphism and let~$\pi^\gridded$ be an~$M$-gridding of~$\pi$. Then reletter and regrid as described above to form a word~$w^\gridded\in(\Sigma^\gridded)^\ast$ and a regridded permutation $\pi^\ggridded$. Under these hypotheses, our previous discussion implies that~$\pi^\ggridded$ has at most $t(1+2ur)$ columns and at most $u(1+2tr)$ rows. We may remove any empty columns or rows in~$\pi^\ggridded$, and so we may assume that every column and every row of~$\pi^\ggridded$ contains at least one entry.

We now define column and row signs~$(c_k)$ and~$(r_\ell)$. Consider first the~$k$\th column of~$\pi^\ggridded$. By removing empty columns, we have ensured that there is at least one letter that encodes an entry of this column. We set $c_k=1$ if the horizontal reading order of this letter is left-to-right, and $c_k=-1$ if the horizontal reading order of this letter is right-to-left. Proposition~\ref{prop-two-letters-one-direction-column} guarantees that this column sign is well-defined---every letter encoding entries in this column has the same horizontal reading order. We assign row signs analogously, setting $r_\ell=1$ if a letter encoding an entry of the $\ell$\th row has vertical reading order bottom-to-top, and $r_\ell=-1$ if it has vertical reading order top-to-bottom. Proposition~\ref{prop-two-letters-one-direction-column} again shows that this row sign is well-defined.

We now define a~$\zpm$~matrix~$M^\ggridded$ by~$M^\ggridded(k,\ell)=c_k r_\ell$ for every $k$ and $\ell$. This is a partial multiplication matrix by definition, and it follows from the definition of~$M^\ggridded$ that~$\pi^\ggridded$ is an~$M^\ggridded$-gridding of~$\pi$. The proof will be complete if we can show that~$\pi^\ggridded$ has a geometric realization. To do this, we appeal to Proposition~\ref{prop-geom-realization-consistent}, which shows that~$\pi^\ggridded$ has a geometric realization if and only if the local orders $\{\colorder_k\}_k$ and $\{\roworder_\ell\}_\ell$ arising from $\pi^\ggridded$, $(c_k)$, and $(r_\ell)$ are consistent.

We prove that these local orders are consistent by establishing that $\psi$ corresponds to a common linear extension of all of them. By symmetry, it suffices to consider a pair of entries in the same column, so suppose that for some indices $i_2<i_3$, the entries $\pi(i_2)$ and $\pi(i_3)$ lie together in the~$k$\th column of~$\pi^\ggridded$. We want to show that $\psi(i_2)<\psi(i_3)$ if and only if $c_k=1$. This holds by the definition of $c_k$ if $\pi(i_2)$ and $\pi(i_3)$ are encoded by the same letter of $\Sigma^\gridded$, so we may assume that these entries are encoded by different letters. In this case, we can apply Proposition~\ref{prop-splitting-entries} and its symmetry (as in the proof of Proposition~\ref{prop-two-letters-one-direction-column}) to see that there are indices $i_1$ and $i_4$ satisfying $i_1<i_2<i_3<i_4$ so that $\pi(i_1)$ and $\pi(i_3)$ are encoded by the same letter of $\Sigma^\gridded$ and $\pi(i_2)$ and $\pi(i_3)$ are encoded by the same letter of $\Sigma^\gridded$.

There are two cases, depending on whether $\pi(i_2)$ and $\pi(i_3)$ share a cell in~$\pi^\gridded$. First suppose that $\pi(i_2)$ and $\pi(i_3)$ lie in different cells of~$\pi^\gridded$. Since~$\pi(i_2)$ lies in a different cell of $\pi^\gridded$ from $\pi(i_1)$ and $\pi(i_3)$, it follows that $\pi(i_2)$ separates these two entries. Thus Proposition~\ref{prop-distinguish-perms} shows that $\psi(i_1)<\psi(i_2)<\psi(i_3)$ or the reverse. Which of these inequalities we have is determined by the horizontal reading order of the letter that encodes $\pi(i_2)$ and $\pi(i_4)$. Since that reading order is encoded in $c_k$, it follows that $\psi(i_2)<\psi(i_3)$ if and only if~${i_2\colorder_k i_3}$, which is what we desire in this case.

Now suppose that $\pi(i_2)$ and $\pi(i_3)$ lie in the same cell of~$\pi^\gridded$, although note that this does not necessarily mean that they share the same cell of~$\pi^\ggridded$. Because~$\pi$ has only trivial monotone intervals, $\pi(i_2)$ and $\pi(i_3)$ must be separated by some entry~$\pi(x)$ that lies outside their common cell in~$\pi^\gridded$. Because all of the entries $\pi(i_1)$, $\pi(i_2)$, $\pi(i_3)$, and $\pi(i_4)$ lie in the same cell of~$\pi^\gridded$, this means that~$\pi(x)$ also separates $\pi(i_1)$ from $\pi(i_3)$ and $\pi(i_2)$ from $\pi(i_4)$. Therefore we may apply Proposition~\ref{prop-distinguish-perms} to see that if $c_k=1$, then both $\psi(i_1)<\psi(x)<\psi(i_3)$ and $\psi(i_2)<\psi(x)<\psi(i_4)$, which implies that $\psi(i_2)<\psi(i_3)$, or if $c_k=-1$, then the reverse holds. In other words, as in the other case, $\psi(i_2)<\psi(i_3)$ if and only if $i_2\colorder_k i_3$.

This shows that the column orders are all consistent with $\psi$. By a symmetrical argument, it follows that the row orders are also consistent with $\psi$, and thus that the column and row orders are consistent with each other. This guarantees, via Proposition~\ref{prop-geom-realization-consistent}, that~$\pi^\ggridded$ has a geometric realization, and thus completes the proof.
\end{proof}

\subsection{From one permutation to the entire class}
\label{subsec-end-of-proof}

We now conclude our proof of Theorem~\ref{thm-ggc-lettericity}. Suppose we have a permutation class $\C$ for which the corresponding graph class $G_{\C}$ has lettericity at most~$r$. We want to show that $\C$ is contained in some geometric grid class. By Corollary~\ref{cor-mono-griddable-graphs}, we know that $\C$ is monotonically griddable, so $\C\subseteq\Grid(M)$ for some~$\zpm$~matrix~$M$. Suppose that the matrix~$M$ has size~$t\times u$.

Choose an arbitrary permutation $\pi\in\C$ of length $n$. As we remarked in Section~\ref{subsec-mono-intervals}, it suffices to consider the case where~$\pi$ has only trivial monotone intervals. Because the lettericity of $G_\C$ is at most~$r$, there is some alphabet $\Sigma$ of size at most~$r$, some decoder $D\subseteq\Sigma^2$, and some word~$w\in\Sigma^n$ such that $G_\pi\cong \Gamma_D(w)$. We therefore have the hypotheses of Proposition~\ref{prop-one-perm-final-result}, so there is some~$\zpm$~matrix~$M_\pi$ of size at most $t(1+2ur)\times u(1+2tr)$ for which $\pi\in\Geom(M_\pi)$.

\interfootnotelinepenalty=10000

Therefore for every permutation $\pi\in\C$ with only trivial monotone intervals, there is a~$\zpm$~matrix~$M_\pi$ of size at most $t(1+2ur)\times u(1+2tr)$ for which $\pi\in\Geom(M_\pi)$. Thus there is some finite~$\zpm$~matrix~$M$ that contains all such matrices~$M_\pi$ as submatrices%
\footnote{In fact, this step of the proof does not require much expansion of these matrices. As in Section~\ref{sec-ggc-encoding} where we observed that $\Geom(M)=\Geom(M^{\times 2})$ for every~$\zpm$~matrix~$M$, if we replace every entry by $\tinymatrix{-1&1\\1&-1}$, it follows that for every~$\zpm$~matrix~$M$ of size~$t\times u$, we have $\Geom(M)\subseteq \Geom(S_{t,u})$ where $S_{t,u}$ is the $2t\times 2u$ matrix defined by $S_{t,u}(k,\ell)=(-1)^{k+\ell-1}$. One therefore only needs to consider a large enough matrix of this form.}, and thus we have $\C\subseteq\Geom(M)$, completing the proof of Theorem~\ref{thm-ggc-lettericity}.

\section{Concluding Remarks}
\label{sec-conclusion}

We have proved that, when restricted to classes of inversion graphs, bounded lettericity is equivalent to geometric griddability of the corresponding permutation class. Several intriguing questions about both of these concepts remain. One area that is currently under active investigation is a characterization of the geometrically griddable permutation classes; Theorem~\ref{thm-mono-griddable} completely characterizes the \emph{monotonically} griddable permutation classes, but as yet it has no geometric analogue. Similarly, there is no known characterization of graph classes of bounded lettericity, and the results established here suggest that it may be possible to answer both of these questions simultaneously.

Even if these characterizations remain out of reach, a decision procedure to determine if a given class (of permutations or graphs, specified by their minimal forbidden objects) is geometrically griddable or of bounded lettericity would be desirable.

There is an ongoing effort to develop robust tools for tackling these problems. In~\cite{alecu:understanding-l:}, Alecu and Lozin introduce the notion of \emph{locally ordered hypergraphs}: hypergraphs enhanced with linear orders on their edges, subject to a local consistency condition that the orders should agree where they overlap. This attempts to abstract the settings of monotone griddings and that of letter graph expressions. A proof of Theorem~\ref{thm-ggc-lettericity} with a similar outline as ours, but using the machinery of locally ordered hypergraphs is presented in Alecu's thesis~\cite{alecu:a-parametric-ap:}.

%
%
%
%
%
%
%
%
%
%
%
%
%
%
%
%
%
%
%
%
%
%
%
%
%
%
%
%
%
%
%
%
%
%
%
%
%
%
%
%
%
%
%
%
%
%
%
%
%
%
%
%
%
%


\def\cprime{$'$}

\end{document}